\documentclass[11pt]{article}
\usepackage{amsmath, amssymb, amsthm, verbatim,enumerate,bbm,color} %DIF >
\usepackage{indentfirst}
\usepackage[vmargin=1in,hmargin=1in]{geometry}
\usepackage{graphicx}
\usepackage{bbm}
\usepackage{tikz}

\title{The Removal Lemma for Tournaments}

\author{
Jacob Fox \thanks {
Stanford University, Stanford, CA. Supported by a Packard Fellowship, by NSF CAREER award DMS 1352121, and by an Alfred P. Sloan Fellowship. Email: jacobfox@stanford.edu.
}
\and	Lior Gishboliner \thanks{School of Mathematics, Tel Aviv University, Tel Aviv 69978, Israel.
Email: liorgis1@post.tau.ac.il.}
\and Asaf Shapira \thanks{
School of Mathematics, Tel Aviv University, Tel Aviv 69978, Israel.
Email: asafico$@$tau.ac.il. Supported in part by ISF Grant 1028/16 and ERC-Starting Grant 633509.
}
\and Raphael Yuster
\thanks{Department of Mathematics, University of Haifa, Haifa 31905, Israel. Email: raphy$@$math.haifa.ac.il. Supported in part by ISF Grant 1028/16.}}

\date{\today}
\allowdisplaybreaks
\expandafter\let\expandafter\oldproof\csname\string\proof\endcsname
\let\oldendproof\endproof
\renewenvironment{proof}[1][\proofname]{%
  \oldproof[\bf #1]%
}{\oldendproof}

\theoremstyle{plain}
\newtheorem{theorem}{Theorem}[section]
\newtheorem{lemma}[theorem]{Lemma}
\newtheorem{claim}[theorem]{Claim}
\newtheorem{proposition}[theorem]{Proposition}

\newtheorem{corollary}[theorem]{Corollary}

\newtheorem{definition}[theorem]{Definition}

\newcommand{\Bin}{\ensuremath{\textrm{Bin}}}

\newcommand{\C}{\mathcal C}

\newcommand{\K}{\mathcal K}
\newcommand{\N}{\mathcal N}
\newcommand{\Q}{\mathcal Q}
\newcommand{\U}{\mathcal U}
\newcommand{\W}{\mathcal W}
\newcommand{\Image}{\text{Im}}
\newcommand{\homleq}{\leq_{\hom}}
\RequirePackage[normalem]{ulem} %DIF PREAMBLE
\RequirePackage{color}\definecolor{RED}{rgb}{1,0,0}\definecolor{BLUE}{rgb}{0,0,1} %DIF PREAMBLE
\begin{document}
\maketitle

\begin{abstract}

Suppose one needs to change the direction of at least $\epsilon n^2$ edges of an $n$-vertex tournament $T$, in order to make it $H$-free.
A standard application of the regularity method shows that in this case $T$ contains at least $f^*_H(\epsilon)n^h$ copies of $H$, where $f^*_H$ is some tower-type function. It has long been observed that many graph/digraph problems become easier when assuming that
the host graph is a tournament. It is thus natural to ask if the removal lemma becomes easier if we assume that the digraph $G$ is a tournament.

Our main result here is a precise characterization of the tournaments $H$ for which $f^*_H(\epsilon)$ is polynomial in $\epsilon$, stating
that such a bound is attainable if and only if $H$'s vertex set can be partitioned into two sets, each spanning an acyclic directed graph.
The proof of this characterization relies, among other things, on a novel application of a regularity lemma for matrices due to Alon, Fischer and Newman,
and on probabilistic variants of Ruzsa-Szemer\'edi graphs.

We finally show that even when restricted to tournaments, deciding if $H$ satisfies the condition of our characterization is an NP-hard problem.

\end{abstract}

\section{Introduction}

\subsection{Background and motivation}

Suppose an $n$-vertex graph $G$ contains $cn^h$ copies of an $h$-vertex graph $H$.
It is clear that in this case one should remove at least $c'n^2$ edges in order to turn $G$ into an $H$-free graph.
The celebrated removal lemma of Ruzsa and Szemer\'edi \cite{RS} states that (at least qualitatively) this sufficient
condition is in fact necessary. More precisely, it states that there is a function $f_H(\epsilon)$ so that if one
needs to remove at least $\epsilon n^2$ edges from an $n$-vertex graph $G$ in order to make it $H$-free, then $G$ contains
at least $f_H(\epsilon)n^h$ copies of $H$. Besides its intrinsic interest, the removal lemma was extensively studied also due
to its many applications. See \cite{CF} for more background on the lemma and its many variants.

All proofs of the removal lemma apply some version of Szemer\'edi's regularity lemma \cite{Sz}, and thus can only bound $f_H(\epsilon)$ by tower-type
functions of $\epsilon$. It is a major open problem in extremal graph theory to decide if a non-tower-type bound can be obtained even for the special
case of $H=K_3$. Given the above, it is thus natural to ask for which graphs $H$ one can obtain very efficient bounds for $f_H(\epsilon)$. The first result of this type was obtained by Alon \cite{subgraphs} who proved that $f_H(\epsilon)$ is polynomial in $\epsilon$ if and only if $H$ is a bipartite graph.
Alon and Shapira \cite{directed} considered the analogous question in the setting of directed graphs and proved that $f_H(\epsilon)$ is polynomial if and only if $H$ has a homomorphism into an oriented tree or a $2$-cycle, where an {\em oriented graph} $H = (V,E)$ is an orientation of an undirected graph, that is, a directed graph in which, for every pair of distinct vertices $x,y \in V$, there is at most one edge between $x$ and $y$.

Our focus in this paper is in studying analogous questions in the setting of tournaments. The precise definition is the following:
Suppose $H$ is a fixed oriented graph. We say that an $n$-vertex tournament $T$ is $\epsilon$-far from being $H$-free if one should
change the direction of at least $\epsilon n^2$ edges in order to turn $T$ into an $H$-free tournament. It is not hard to apply the regularity
method, in a way similar to \cite{AFKS}, and show that if $T$ is $\epsilon$-far from being $H$ free then $T$ contains $f^*_H(\epsilon)n^h$ copies of
$H$, where again $f^*_H(\epsilon)$ is a tower-type function. The question we are interested in is then for which $H$ can $f^*_H(\epsilon)$ be bounded
by a polynomial in $\epsilon$?

As is well known, tournaments possess many properties not shared by general oriented graphs. As a result,
many problems that are hard to resolve in general oriented graphs, become easier for tournaments.
It is thus natural to ask if the removal lemma is easier for tournaments? Let us mention that it is known that
there is at least one $H$ for which the removal lemma is known to be easier for tournaments. Indeed, it follows from the result of \cite{directed} that $f_{C_3}(\epsilon)$ is not polynomial \footnote{This special case of the result of \cite{directed} is actually implicit already in \cite{RS}} in $\epsilon$, while Fox and Sudakov \cite{FS} proved that $f^*_{C_3}(\epsilon)$ is polynomial in $\epsilon$.

Let us conclude by mentioning that a further motivation for this paper was the work of Berger et al. \cite{BCCF} on tournaments
they called {\em Heroes}. See \cite{SS} for more details. The work of \cite{BCCF} is another nice example of a phenomenon that
holds in tournaments but fails to hold for general digraphs. One of the notions studied in \cite{BCCF} is the chromatic number
of a tournament $T$ defined as the smallest number of transitive tournaments which cover $V(T)$. As Theorem \ref{thm:NP_hard} shows,
this notion is also relevant in our setting.

\subsection{Our main results}

Our main result in this paper gives a precise characterization of the oriented graphs $H$ for which one can prove
a removal lemma in tournaments with a polynomial bound. Let us say that an oriented graph $H$ is {\em easy} if there is
a constant $c = c(H)$ satisfying $f^*_H(\epsilon) \geq \epsilon^c$ for every sufficiently small $\varepsilon > 0$. If $H$ is not easy then it is {\em hard}.

\begin{theorem}\label{thm:main}
$H$ is easy if and only if $V(H)$ can be partitioned into
$2$ vertex sets, each spanning an acyclic directed graph.
\end{theorem}

It was shown in \cite{directed} that an oriented graph $H$ satisfies $f_H(\epsilon) \geq \epsilon^c$ in general digraphs, only if $H$ has
a homomorphism into an oriented tree. Observe that Theorem \ref{thm:main} shows that $f_H^*(\epsilon)$ is polynomial
for a much wider class of oriented graphs, that is, there is an entire family of oriented graphs for which the removal lemma
is easier for tournaments.

We believe that the proofs of both directions of Theorem \ref{thm:main} are of independent interest.
We note that the analogous ``if'' parts in the characterizations given in \cite{subgraphs,directed} followed
from simple density/Tur\'an type arguments. For example, the fact that a bipartite $H$ is easy in undirected
graphs, follows from the simple reason that a graph with $\epsilon n^2$ edges contains at least $\epsilon^{h^2}n^{h}$ copies of $H$.
In contrast, our proof requires a much more elaborate argument: we first show that for every oriented graph $H$ as in the theorem
there is an oriented complete bipartite graph, so that no matter how one completes this bipartite digraph into a tournament,
one always ends up with a tournament containing a copy of $H$ (see Lemma \ref{lem:bipartite_force}).
We then combine this with a novel application of an efficient ``conditional regularity lemma''
for matrices of Alon, Fischer and Newman \cite{bipartite} (see also \cite{FPS} and \cite{LS} for related results) in order to complete the proof.
As to the ``only if'' part, as in previous
lower bounds for removal lemmas, we also make use of variants of Ruzsa-Szemer\'edi \cite{RS}
graphs. Our construction however, requires several additional twists such as the notion of ordered homomorphisms defines in Section \ref{sec:hard},
and the probabilistic construction from Section \ref{sec:lemma}.

Let us conclude by describing our final result.
It is natural to ask if the characterization given in Theorem \ref{thm:main} is ``efficient'',
that is, how hard is it to tell if an oriented graph $H$ is easy. It follows from the work of Bokal et al. \cite{BFJKM} that
this task is in fact $NP$-hard. Continuing with the theme of studying whether problems become easier when restricted
to tournaments, it is natural to ask if one can at least recognize tournaments whose vertex set can be partitioned
into $2$ sets, each spanning an acyclic directed graph, i.e. into two transitive tournaments.
The following theorem strengthens the result of Bokal et al \cite{BFJKM} by showing that the problem is hard even for tournaments.

\begin{theorem}\label{thm:NP_hard}
For every $k \geq 2$, the problem of deciding if a tournament is $k$-colorable is $NP$-hard.
\end{theorem}

\subsection{Organization}
In Section \ref{sec:lemma} we describe a probabilistic construction that will be crucial
in the proofs of both directions of Theorem \ref{thm:main}. The proof of the first direction of Theorem \ref{thm:main} is
given in Section \ref{sec:easy}, while the second is given in Section \ref{sec:hard}. The proof of Theorem \ref{thm:NP_hard} is given in Section \ref{sec:np}.

\section{A Preliminary Lemma}\label{sec:lemma}

A {\em proper $k$-coloring} (or simply {\em $k$-coloring}) of an oriented graph $H$ is a partition of $V(H)$ into $k$ sets, each inducing an acyclic digraph. We say that $H$ is {\em $k$-colorable} if it has a proper $k$-coloring. Notice that if $H$ is a tournament then this definition coincides with the definition of a $k$-colorable tournament.

In this section we prove Lemma \ref{lem:force}, stated below, which will be a key ingredient in the proof
of both directions of Theorem \ref{thm:main}. We start with some notation which we will also use in later sections.
For a pair of vertices $x,y \in V$, we write $x \rightarrow y$ to mean that $(x,y) \in E$. For a pair of disjoint subsets $X,Y \subseteq V$, we use the notation $X \rightarrow Y$ to mean that there is no $(x,y) \in X \times Y$ for which $y \rightarrow x$. In other words,
$X \rightarrow Y$ means that for every $(x,y) \in X \times Y$, either
$x \rightarrow y$ or there are no edges between $x$ and $y$.\footnote{This definition might seem strange as, for example, $x \rightarrow y$ is not the same as $\{x\} \rightarrow \{y\}$. Nevertheless, this definition is useful and should not cause confusion.}
Evidently, if the digraph is a tournament then
$X \rightarrow Y$ is equivalent to saying that $x \rightarrow y$ for every $(x,y) \in X \times Y$.
For a digraph $G$ and a set $X \subseteq V(G)$, we use $G[X]$ to denote the subdigraph of $G$ induced by $X$.

A {\em $k$-partite tournament} is an orientation of a complete $k$-partite graph. Notice that a bipartite tournament (i.e., a $k$-partite tournament for $k=2$) is not the same as a $2$-colorable tournament. A
{\em completion} of a $k$-partite tournament
$F = (V_1 \cup V_2 \dots \cup V_k, E)$ is any tournament on $V(F)$ that agrees with $F$ on the edges between the sets $V_1,...,V_k$, i.e. any tournament obtained from $F$ by adding $k$ arbitrary tournaments on the sets $V_1,\dots,V_k$.
%For a completion $T$ of $F$, we say that $e \in E(T)$ is a {\em cut-edge} if $e \in E(F)$, i.e. $e$ goes between the sets $V_1,...,V_k$.

\begin{lemma}\label{lem:force}
For every $h \geq 2$ there are $m_0 = m_0(h)$ and
$\gamma = \gamma(h) > 0$ with the following property. Let $H$ be an oriented graph on $h$ vertices and let $D$ be an oriented graph on $[k]$, where
$2 \leq k \leq h$. Suppose that $H$ has a proper $k$-coloring,
$V(H) = H_1 \cup \dots \cup H_k$, such that $H_i \rightarrow H_j$ for every $(i,j) \in E(D)$. Then for every
$m \geq m_0$ there is a $k$-partite tournament
$F = (V_1 \cup \dots \cup V_k, E(F))$ such that:
\begin{enumerate}
\item $|V_i| = m$ for every $i = 1,...,k$.
% \item For every $1 \leq i \neq j \leq k$, if $(i,j) \in E(D)$ then $V_i \rightarrow V_j$.
\item $V_i \rightarrow V_j$ for every $(i,j) \in E(D)$.
\item Every completion of $F$ contains a collection of at least
$\gamma m^2$ copies of $H$ with the property that every edge
$e \in E(F)$ is contained in at most one of these copies.
\end{enumerate}
\end{lemma}
\noindent
In the proof of Lemma \ref{lem:force} we use the following three claims.
Denote by $\text{Bin}(N,p)$ the binomial distribution with parameters $N$ and $p$.
We will need the following standard Chernoff-type bound.
\begin{claim}[\cite{Prob_Method}]\label{claim:Chernoff}
$\Pr \Big[ \text{Bin}(N,p) < (1-\alpha)Np  \Big] \leq e^{-Np\alpha^2/2}$\;.
\end{claim}
\noindent
The following claim is a well-known fact from Ramsey Theory.
\begin{claim}[\cite{Moon}]\label{claim:Ramsey}
Every tournament on $2^{k-1}$ vertices contains a transitive subtournament on $k$ vertices.
\end{claim}
\noindent

\begin{claim}\label{claim:disj_collection}
Let $t,k \geq 1$ be integers. Then there is a collection
$\mathcal{S} \subseteq [t]^{k}$ of size at least $t^2/k^2$ such that every pair of distinct $k$-tuples in $\mathcal{S}$ have at most one identical entry.
\end{claim}
\begin{proof}
We construct the collection $\mathcal{S}$ greedily: we start with an empty collection, add an arbitrary $k$-tuple to it, discard all $k$-tuples that coincide in more than one entry with the $k$-tuple we added and repeat. At the beginning we have all $t^k$ of the $k$-tuples in $[t]^k$. At each step we discard at most $\binom{k}{2}t^{k-2}$ tuples.
Therefore, at the end of the process we have a collection of size at least
\begin{equation*}
\frac{t^k}{1 + \binom{k}{2}t^{k-2}} \geq \frac{t^k}{k^2t^{k-2}} = \frac{t^2}{k^2},
\end{equation*}
as required.
\end{proof}
\begin{proof}[Proof of Lemma \ref{lem:force}]
For every $i = 1,...,k$ put $h_i = |H_i|$. Fix an integer $m > m_0(h)$, where $m_0(h)$ will be chosen later.
For convenience of presentation, we assume that $m$ is divisible by $2h$ and by $2h_i$ for every $i$.
Let $V_1,...,V_k$ be pairwise-disjoint vertex sets of size $m$ each. The edges between the sets $V_1,...,V_k$ are oriented as follows: for every
$(i,j) \in E(D)$ we direct all edges from $V_i$ to $V_j$. For every $1 \leq i < j \leq k$ for which $(i,j),(j,i) \notin E(D)$, orient the edges between $V_i$ and $V_j$ randomly and independently with probability $1/2$. We will show that with positive probability, the resulting $k$-partite tournament, $F$, satisfies the assertion of Item 3 in the lemma, thus finishing the proof.

\noindent
An {\em $H$-partition} is a tuple $\left( \mathcal{P}_{i,j}, \mathcal{T}_{i,j} \right)_{i,j}$, where $1 \leq i \leq k$ and $1 \leq j \leq \frac{m}{2h_i}$, with the following properties.
\begin{itemize}
\item For each $1 \leq i \leq k$,
$\mathcal{P}_{i,1},...,\mathcal{P}_{i,\frac{m}{2h_i}}$ are pairwise-disjoint subsets of $V_i$, each of size $h_i = |H_i|$.
\item For each $1 \leq i \leq k$ and $1 \leq j \leq \frac{m}{2h_i}$,
$T_{i,j}$ is a labeled transitive tournament on the set $\mathcal{P}_{i,j}$.
% for every $i = 1,...,k$ and $\ell \in \left[ \frac{m}{h} \right]$
\end{itemize}
Note that $\bigcup_{j=1}^{m/2h_i}{\mathcal{P}_{i,j}}$ is a subset of $V_i$ of size exactly $\frac{m}{2}$.
The number of ways to choose an $H$-partition is exactly
\begin{equation}\label{eq:H_partition_count}
\prod_{i=1}^{k}{\frac{m!}{(m/2)!}} \leq m^{km}.
\end{equation}

% $$\left( \frac{m!}{(m/h)!} \right)^{k}
% \prod_{i=1}^{k}{\left( \frac{1}{(h-h_i)!} \right)^{m/h}} \leq m^{km}$$

\noindent
By Claim \ref{claim:disj_collection} with parameter $t = \frac{m}{2h}$, there is a collection $\mathcal{S} \subseteq \left[ \frac{m}{2h} \right]^k \subseteq
\big[\frac{m}{2h_1} \big]
\times \dots \times
\big[\frac{m}{2h_k} \big]$ such that
$\left| \mathcal{S} \right| \geq \left( \frac{m}{2hk} \right)^2
\geq \frac{m^2}{4h^4}$,
%\begin{equation}\label{eq:disjoint_collection_size}
%\left| \mathcal{S} \right| \geq \left( \frac{m}{2hk} \right)^2
%\geq \frac{m^2}{4h^4},
%\end{equation}
and
\begin{equation}\label{eq:disjoint_collection_property}
\forall s=(s_1,...,s_k), s'=(s'_1,...,s'_k) \in \mathcal{S}, \; \; \;
\#\left\{ 1 \leq i \leq k : s_{i} = s'_{i}\right\} \leq 1.
\end{equation}

For each $i = 1,\dots,k$ we fix a linear ordering of the vertices of $H_i$ in which all edges point forward, that is, if $u,v \in H_i$ and $u \rightarrow v$ then $u$ precedes $v$ in the ordering. Such an ordering exists since $H_i$ is acyclic. Fix an $H$-partition
$\Q = \left( \mathcal{P}_{i,j}, \mathcal{T}_{i,j} \right)_{i,j}$ and let
$s = (s_1,...,s_k) \in \mathcal{S}$. Since $T_{i,s_i}$ is transitive and $H_i$ is acyclic, $H_i$ can be embedded into $T_{i,s_i}$. In what follows, when we say that $T_{i,s_i}$ plays the role of $H_i$ we mean that $H_i$ is embedded in $T_{i,s_i}$ in an order-preserving way with respect to our fixed linear ordering of $H_i$ and the unique linear ordering of $T_{i,s_i}$. Let $A_{\Q}(s)$ be the event that $T_{1,s_1} \cup \dots \cup T_{k,s_k}$, together with the edges of $F$ connecting the sets $\mathcal{P}_{1,s_1},\dots,\mathcal{P}_{k,s_k}$, contains a copy of $H$ with $T_{i,s_i}$ playing the role of $H_i$. Then
$\mathbb{P} \left[ A_{\Q}(s) \right] \geq
2^{-\sum{h_ih_j}} \geq 2^{-h^2}$.
%\begin{equation}\label{eq:event_prob}
%\mathbb{P} \left[ A_{\Q}(s) \right] \geq
%2^{-\sum{h_ih_j}} \geq 2^{-h^2}.
%\end{equation}
Observe that by (\ref{eq:disjoint_collection_property}), the events $\{A_{\Q}(s) : s \in \mathcal{S}\}$ are independent. Since
$\left| \mathcal{S} \right| \geq \frac{m^2}{4h^4}$, the random variable
% $$X = \#\left\{ s \in \mathcal{S} : A_{\Q}(s) \text{ happened} \right\}$$
$$ X_{\Q} = \sum\limits_{s \in \mathcal{S}}{ \mathbbm{1}_{A_{\Q}(s)}} $$
stochastically dominates a random variable with distribution
$\Bin\left( \frac{m^2}{4h^4}, 2^{-h^2} \right)$. By Claim \ref{claim:Chernoff} with parameter
$\alpha = \frac{1}{2}$ we have:
\begin{equation*}
\mathbb{P}\left[ X_{\Q} < \frac{2^{-h^2}m^2}{8h^4} \right] \leq
\mathbb{P}\left[ \Bin\left( \frac{m^2}{4h^4}, 2^{-h^2} \right) < \frac{2^{-h^2}m^2}{8h^4} \right] \leq
\exp\left\{ - \frac{m^2}{32h^42^{h^2}} \right\} < m^{-hm} \leq m^{-km}.
\end{equation*}
The strict inequality above holds if $m$ is large enough. We choose $m_0(h)$ to be large enough so that this inequality holds for every $m \geq m_0(h)$. Set
$\gamma = \gamma(h) = \frac{2^{-h^2}}{8h^4}$.
By (\ref{eq:H_partition_count}), there are at most $m^{km}$ ways to choose an $H$-partition $\Q$. By the union bound over all $H$-partitions we get that the following event has positive probability: for every $H$-partition $\Q$, the number of
$s \in \mathcal{S}$ for which $A_{\Q}(s)$ happened is at least
$\gamma m^2$. We now show that if this event happens then $F$ satisfies the assertion of Item 3 in the lemma.

Let $T$ be a completion of $F$. For every $1 \leq i \leq k$, we use Claim \ref{claim:Ramsey} to extract from $V_i$ a collection
$\mathcal{P}_{i,1},...,\mathcal{P}_{i,\frac{m}{2h_i}}$ of pairwise-disjoint sets, each of size $h_i$, such that $T[\mathcal{P}_{i,j}]$ is transitive for every
$1 \leq j \leq \frac{m}{2h_i}$. We extract these sets one by one and stop when there are
$\frac{m}{2}$ remaining vertices. By Claim \ref{claim:Ramsey}, we can do this as long as there are at least $2^{h_i-1}$ remaining vertices. By choosing $m_0(h)$ to be large enough we guarantee that $\frac{m}{2} \geq 2^{h_i-1}$.

For every $1 \leq i \leq k$ and $1 \leq j \leq \frac{m}{2h_i}$, set
$T_{i,j} = T[\mathcal{P}_{i,j}]$.
Consider this $H$-partition $\Q = \left( \mathcal{P}_{i,j}, \mathcal{T}_{i,j} \right)_{i,j}$.
%Let $\mathcal{S}_0$ be the set of all
%$s = (s_1,...,s_k) \in \mathcal{S}$ for which $A_{\Q}(s)$ happened.
%By our assumption we have $\left| \mathcal{S}_0 \right| \geq \gamma m^2$.
By our assumption, the event $A_{\Q}(s)$ happened for at least $\gamma m^2$ of the elements
$s \in \mathcal{S}$.
By the definition of the event $A_{\Q}(s)$, if this event happened then the vertex-set $\mathcal{P}_{1,s_1} \cup \dots \cup \mathcal{P}_{k,s_k}$
contains a copy of $H$ (in the tournament $T$) with $\mathcal{T}_{i,s_i}$ playing the role of $H_i$. By (\ref{eq:disjoint_collection_property}), every pair of these copies can share vertices in no more than one of the clusters $V_1,\dots,V_k$. Therefore, every edge $e \in E(F)$ (that is, an edge that connects vertices in two distinct clusters $V_i$, $V_j$) is contained in at most one of these copies. Thus, Item 3 in the lemma holds, completing the proof.
\end{proof}

\section{Easy Tournaments}\label{sec:easy}

\noindent
In this section we prove the first part of Theorem \ref{thm:main}.
For convenience, we restate as follows.

\begin{theorem}\label{thm:easy}
For every $h$ there are $\varepsilon_0 = \varepsilon_0(h) > 0$ and
$d = d(h)$ with the following property. For every $2$-colorable oriented graph $H$ on $h$ vertices and for every positive
$\varepsilon < \varepsilon_0$, if a tournament $T$ on
$n \geq n_0(\varepsilon)$ vertices is $\varepsilon$-far from being $H$-free then $T$ contains at least $\varepsilon^dn^h$ copies of $H$.
\end{theorem}
Throughout this section, we implicitly assume that $n$ is large enough. To make the presentation cleaner, we also implicitly assume that $n$ is divisible by various quantities which depend on the other parameters, $H$ and $\varepsilon$. It is easy to see that in order to establish Theorem \ref{thm:easy}, it is enough to prove it for values of $n$ which satisfy such divisibility conditions.

We start by introducing some definitions and lemmas that we use in the proof of Theorem \ref{thm:easy}.
Let $T$ be a tournament on $[n]$.
The {\em adjacency matrix} of $T$, denoted $A = A(T)$, is the $n \times n$ matrix in which, for every $1 \leq i \neq j \leq n$, $A_{i,j} = 1$ if $(i,j) \in E(T)$ and $A_{i,j} = 0$ if $(j,i) \in E(T)$. The main diagonal of $A$ is set to be $0$. For a pair of disjoint sets $X,Y \subseteq V(T)$ define
\begin{equation*}
e(X,Y) = \left| \left\{(x,y) \in X \times Y : (x,y) \in E(T)\right\} \right|
\end{equation*}
and $d(X,Y) = \frac{e(X,Y)}{|X||Y|}$. Note that $d(X,Y) + d(Y,X) = 1$, as $T$ is a tournament.
We have $X \rightarrow Y$ if and only if $d(X,Y) = 1$, and $Y \rightarrow X$ if and only if
$d(X,Y) = 0$. For a constant $\delta < \frac{1}{2}$, we say that $(X,Y)$ is $\delta$-homogeneous if either
$d(X,Y) \geq 1 - \delta$ or $d(X,Y) \leq \delta$.
We say that the {\em dominant direction} of $(X,Y)$ is $X \rightarrow Y$ if $d(X,Y) \geq \frac{1}{2}$ and is $Y \rightarrow X$ if
$d(X,Y) < \frac{1}{2}$.
The {\em weight} of the pair $(X,Y)$ is $\frac{|X||Y|}{n^2}$.
Let $\mathcal{P} = \{V_1,...,V_r\}$ be a vertex-partition of $T$, namely suppose that
$V(T) = V_1 \uplus \dots \uplus V_r$.
We say that $\mathcal{P}$ is {\em $\delta$-homogeneous} if the total weight of non-$\delta$-homogeneous pairs
$(V_i,V_j)$, $1 \leq i \neq j \leq r$, is at most $\delta$.
% We say that $\mathcal{P}$ is a {\em $(\delta,r)$-partition} if it is $\delta$-homogeneous and $\left|  \mathcal{P}\right| \leq r$.
We say that $\mathcal{P}$ is an {\em equipartition} if
$\big| |V_i| - |V_j| \big| \leq 1$ for every $1 \leq i,j \leq r$.

Recall the definition of a bipartite tournament from Section 2. A copy of a bipartite tournament
$F = (M \cup N, E)$ in a tournament $T$ is an injection $f : V(F) \rightarrow V(T)$ such that for every $x \in M, y \in N$ it holds that $(x,y) \in E(F)$ if and only if
$(f(x),f(y)) \in E(T)$. The first ingredient in the proof of Theorem \ref{thm:easy} is the following lemma.

\begin{lemma}\label{lem:strong_reg}
There is $C > 0$ such that the following holds for every integer $k \geq 1$ and every
$\delta \in (0,1/2)$. Let $F = (M \cup N, E)$ be a bipartite tournament with $|M| = |N| = k$. Then every tournament $T$ on $n \geq n_0(k,\delta)$ vertices either contains at least
$\left( \delta/k \right)^{Ck^3} n^{2k}$ copies of $F$ or satisfies the following: there is an equipartition
$\mathcal{Q} = \{Q_1,...,Q_q\}$ of $V(T)$, where
$q \geq \frac{1}{\delta}$, and there are subsets $W_i \subseteq Q_i$, such that the following hold.
\begin{enumerate}
\item For all but at most $\delta q^2$ of the pairs
$1 \leq i < j \leq q$, it holds that $(Q_i,Q_j)$ is $\delta$-homogeneous and the dominant direction of $(W_i,W_j)$ is the same as that of $(Q_i,Q_j)$.
\item $(W_i,W_j)$ is $\delta$-homogeneous for every $1 \leq i < j \leq q$.
\item $|W_i| \geq \left( \delta / k \right)^{Ck^2} n$ for every
$1 \leq i \leq q$.
\end{enumerate}
\end{lemma}

\noindent
Throughout this section, $C$ denotes the constant from Lemma \ref{lem:strong_reg}.
Another ingredient in the proof of Theorem \ref{thm:easy} is the following simple counting lemma.
\begin{lemma}\label{lem:count}
For every $h$ there are $\eta = \eta(h)$ and
$\alpha = \alpha(h) > 0$ such that the following holds for every oriented graph $H$ on $h$ vertices. Let $H_1,...,H_{\ell}$ be a partition of $H$ such that $H_1,...,H_{\ell}$ induce acyclic digraphs, and for every $1 \leq i < j \leq \ell$, either $H_i \rightarrow H_j$ or $H_j \rightarrow H_i$. Let $W_1,...,W_{\ell}$ be pairwise-disjoint vertex sets in a tournament $T$ having the following properties:
\begin{enumerate}
	\item $|W_i| \geq 2^{h-1}$ for every $1 \leq i \leq \ell$.
	\item For every $1 \leq i \neq j \leq \ell$, if $H_i \rightarrow H_j$ then
	$d(W_i,W_j) \geq 1 - \eta$.
\end{enumerate}
Then $T$ contains at least
$\alpha \cdot \prod_{i=1}^{\ell}{|W_i|^{h_i}}$ copies of $H$, where
$h_i = |H_i|$.
\end{lemma}

Recall the definition of a completion of a bipartite tournament from Section $2$. We say that a bipartite tournament $F$ {\em forces} an oriented graph $H$ if every completion of $F$ contains a copy of $H$. The following lemma is the last ingredient we need for the proof of Theorem \ref{thm:easy}.

\begin{lemma}\label{lem:bipartite_force}
For every $2$-colorable oriented graph $H$ there is a bipartite tournament that forces $H$.
\end{lemma}

\begin{proof}[Proof of Theorem \ref{thm:easy}]
Let $H$ be a $2$-colorable oriented graph on $h$ vertices.
% on $h$ vertices.
Apply Lemma \ref{lem:bipartite_force} to get a bipartite tournament
$F = (M \cup N, E)$ that forces $H$. Note that we can clearly assume that $|M| = |N| \geq h$ (by adding additional vertices if necessary). Put
$k := |M| = |N|$.

\noindent
We will prove the theorem with
\begin{equation*}
\varepsilon_0 = \varepsilon_0(h) = \min\left( \frac{1}{3k}, 3\eta(h) \right)
\end{equation*}
and
\begin{equation*}
d = d(h) = 2Ck^3 + \alpha^{-1},
\end{equation*}
where  $\eta(h)$ and $\alpha = \alpha(h)$ are from Lemma \ref{lem:count}, and $C$ is the constant from Lemma \ref{lem:strong_reg}.

Let $\varepsilon < \varepsilon_0$, and let $T$ be any tournament on $n$ vertices which is $\varepsilon$-far from being $H$-free.
% We assume that $n$ is large enough in order to use Lemma \ref{lem:strong_reg}.
Assume first that $T$ contains at least
$\left( \varepsilon/3k \right)^{Ck^3} n^{2k}$ copies of $F$. Since $F$ forces $H$, every copy of $F$ (in a tournament) contains a copy of $H$. Every copy of $H$ is contained in at most $n^{2k-h}$ copies of $F$. Recalling that $\varepsilon < \frac{1}{3k}$, we conclude that $T$ contains at least
$$n^{-(2k-h)} \left( \varepsilon/3k \right)^{Ck^3} n^{2k} =
\left( \varepsilon/3k \right)^{Ck^3} n^h \geq
\varepsilon^{2Ck^3} n^h \geq \varepsilon^d n^h$$
%\begin{equation*}
%% n^{-(2k-h)} \left( \delta/k \right)^{Ck^3} n^{2k} =
%n^{-(2k-h)} \left( \varepsilon/4k \right)^{Ck^3} n^{2k} =
%\left( \varepsilon/4k \right)^{Ck^3} n^h \geq
%\varepsilon^{2Ck^3} n^h \geq \varepsilon^d n^h
%\end{equation*}
copies of $H$, giving the desired result in this case.

Suppose from now on that $T$ contains less than
$\left( \varepsilon/3k \right)^{Ck^3} n^{2k}$ copies of $F$.
We apply Lemma \ref{lem:strong_reg} to $T$ with approximation parameter
$\frac{\varepsilon}{3}$ to get an equipartition
$\mathcal{Q} = \{Q_1,...,Q_q\}$ and subsets $W_i \subseteq Q_i$ with the properties stated in the lemma.
% Assume now that the second alternative in Lemma \ref{lem:strong_reg} holds
Define $\N$ to be the set of pairs
$1 \leq i < j \leq q$ for which either (a) $(Q_i,Q_j)$ is not $\frac{\varepsilon}{3}$-homogeneous, or (b) the dominant direction of $(W_i,W_j)$ is not the same as that of $(Q_i,Q_j)$.
By Lemma \ref{lem:strong_reg} we have $\left| \N \right| \leq \frac{\varepsilon}{3} q^2$. This implies that
	\begin{equation}\label{eq:bad_pairs}
	\sum_{(i,j) \in \N}{|Q_i||Q_j|} \leq
	\frac{\varepsilon}{3} q^2 \left( \frac{n}{q} \right)^2 = \frac{\varepsilon}{3} n^2.
	\end{equation}
\noindent	
Let $T'$ be the tournament obtained from $T$ by making the following changes.
\begin{enumerate}
\item Make $Q_i$ transitive for every $i = 1,\dots,q$.
\item For every $1 \leq i < j \leq q$, if $d(W_i,W_j) \geq 1 - \frac{\varepsilon}{3}$ then set
$Q_i \rightarrow Q_j$ and if $d(W_i,W_j) \leq \frac{\varepsilon}{3}$ then set
$Q_j \rightarrow Q_i$.
\end{enumerate}

By Lemma \ref{lem:strong_reg}, $(W_i,W_j)$ is $\frac{\varepsilon}{3}$-homogeneous for every
$1 \leq i < j \leq q$, so Item 2 covers all options.
The number of edge-reversals made in Item 1 is at most
$q \binom{n/q}{2} < \frac{n^2}{q} \leq \frac{\varepsilon}{3} n^2$. Here we use the inequality
$q \geq \frac{3}{\varepsilon}$, given by Lemma \ref{lem:strong_reg}.
In Item 2, if $(i,j) \notin \N$ then the number of reversals of edges between $Q_i$ and $Q_j$ is at most $\frac{\varepsilon}{3} |Q_i||Q_j|$. Using these facts and (\ref{eq:bad_pairs}) we get that the total number of edge-reversals made in Items 1 and 2 is less than
$\frac{\varepsilon}{3} n^2 + \sum_{i < j}{\frac{\varepsilon}{3} |Q_i||Q_j|} +
\frac{\varepsilon}{3} n^2 \leq
\frac{\varepsilon}{3} n^2 + \frac{\varepsilon}{3} n^2 + \frac{\varepsilon}{3} n^2 = \varepsilon n^2$.
%\begin{equation*}
%\frac{\varepsilon}{4} n^2 + \sum_{i < j}{\frac{\varepsilon}{4} |Q_i||Q_j|} +
%\frac{\varepsilon}{2} n^2 \leq
%\frac{\varepsilon}{4} n^2 + \frac{\varepsilon}{4} n^2 + \frac{\varepsilon}{2} n^2 = \varepsilon n^2.
%\end{equation*}

Since $T$ is $\varepsilon$-far from being $H$-free and since $T'$ is obtained from $T$ by reversing less than $\varepsilon n^2$ edges, $T'$ must contain a copy of $H$. Let $Q_{i_1},...,Q_{i_{\ell}}$ be the parts of $\Q$ which intersect this copy.
For $j = 1,...,\ell$ define $H_{i_j} = H \cap Q_{i_j}$ and
$h_{i_j} = |H_{i_j}|$. From the way we constructed $T'$ from $T$ in Items 1 and 2, it follows that the sets $H_{i_1},...,H_{i_{\ell}}$ are acyclic, and that for every $1 \leq s < t \leq \ell$ we have either
$H_{i_s} \rightarrow H_{i_t}$ or $H_{i_t} \rightarrow H_{i_s}$.
Moreover, for every $1 \leq s \neq t \leq \ell$, if $H_{i_s} \rightarrow H_{i_t}$ then $Q_{i_s} \rightarrow Q_{i_t}$ in $T'$, implying
$d(W_{i_s},W_{i_t}) \geq 1 - \frac{\varepsilon}{3} \geq 1 - \eta(h)$ in $T$ (see our choice of $\varepsilon_0$). Finally, by Lemma \ref{lem:strong_reg} we have
\begin{equation}\label{eq:W_part_bound}
|W_{i_j}| \geq \left( \varepsilon / 3k \right)^{Ck^2} n
\end{equation}
for every $j = 1,\dots,\ell$.
So if $n$ is large enough then $|W_{i_j}| \geq 2^{h-1}$. We conclude that $W_{i_1},\dots,W_{i_{\ell}}$
satisfy the conditions of Lemma \ref{lem:count} {\em in the tournament $T$} with respect to the partition $V(H) = H_{i_1} \cup \dots \cup H_{i_{\ell}}$.
By applying Lemma \ref{lem:count} and using the inequalities (\ref{eq:W_part_bound}) and
$\varepsilon < \frac{1}{3k}$, we get that $T$ contains at least
\begin{equation*}
\alpha \cdot
\prod_{j=1}^{\ell}{ \left| W_{i_j} \right|^{h_{i_j}}} \geq
\alpha \left( \varepsilon / 3k \right)^{Chk^2} n^h \geq
\varepsilon^{1/\alpha} \left( \varepsilon / 3k \right)^{Ck^3} n^h \geq
\varepsilon^{2Ck^3 + 1/\alpha}n^h = \varepsilon^d n^h
\end{equation*}
copies of $H$. This completes the proof of the theorem.
\end{proof}

\noindent
Having proven Theorem \ref{thm:easy}, we proceed to prove Lemmas \ref{lem:strong_reg}, \ref{lem:count} and \ref{lem:bipartite_force}.

\begin{proof}[Proof of Lemma \ref{lem:bipartite_force}]
This lemma follows easily from Lemma \ref{lem:force}. Let
$V(H) = H_1 \cup H_2$ be a proper $2$-coloring of $H$. Apply Lemma \ref{lem:force} with parameter $h = |V(H)|$ and with $D$ being the empty graph on $2$ vertices. Lemma \ref{lem:force} implies that there is a bipartite tournament $F = (V_1 \cup V_2, E(F))$, where $|V_1|=|V_2| = m := m_0(h)$, such that every completion of $F$ contains at least $\gamma(h) m^2$ (and in particular at least one) copies of $H$.
\end{proof}

\begin{proof} [Proof of Lemma \ref{lem:count}]
Set $m = 2^{h-1}$. For each
$i = 1,...,\ell$ we choose a subset
$X_i \subseteq W_i$ of size
$m$ uniformly at random. For $1 \leq i < j \leq \ell$, let us say that $(X_i,X_j)$ {\em agrees} with $(H_i,H_j)$ if $X_i \rightarrow X_j$ whenever $H_i \rightarrow H_j$ and $X_j \rightarrow X_i$ whenever $H_j \rightarrow H_i$. By the assumption on the pairs $(W_i,W_j)$, the probability that $(X_i,X_j)$ does not agree with $(H_i,H_j)$ is at most $\eta m^2$. By the union bound, the probability that there is a pair
$1 \leq i < j \leq \ell$ for which $(X_i,X_j)$ does not agree with $(H_i,H_j)$ is at most
$\eta m^2 \binom{\ell}{2} \leq \eta m^2 h^2$. By setting
$\eta(h) = \frac{1}{2}(hm)^{-2} = \frac{1}{2}h^{-2}2^{-2(h-1)}$ we get that this probability is at most $\frac{1}{2}$.

By Claim \ref{claim:Ramsey} and the choice of $m$ we get that $X_i$ contains a transitive subset $Y_i$ of size $h_i = |H_i|$. Therefore, if $(X_i,X_j)$ agrees with $(H_i,H_j)$ for every $1 \leq i < j \leq \ell$ then
$X = \bigcup_{i=1}^{\ell}{X_i}$ contains a copy of $H$ with $Y_i$ playing the role of $H_i$. Every such copy of $H$ is contained in at most
$\prod_{i=1}^{\ell}{\binom{|W_i| - h_i}{m - h_i}}$ such sets $X$.
Therefore there are at least
\begin{equation*}
\frac{\frac{1}{2}\prod_{i = 1}^{\ell}{\binom{|W_i|}{m}}}
{\prod_{i=1}^{\ell}{\binom{|W_i| - h_i}{m - h_i}}} =
\frac{\frac{1}{2}\prod_{i=1}^{\ell}{\binom{|W_i|}{h_i}}}
{\prod_{i=1}^{\ell}{\binom{m}{h_i}}} \geq
\frac{1}{2} \cdot \prod_{i=1}^{\ell}{\left( \frac{|W_i|}{m}\right)^{h_i}}
=
\frac{1}{2} \cdot m^{-h} \prod_{i=1}^{\ell}{|W_i|^{h_i}}
\end{equation*}
copies of $H$. We choose the constant $\alpha = \alpha(h)$ to be
$\alpha = \frac{1}{2}m^{-h} = 2^{-h(h-1)-1}$.
%\begin{equation*}
%\frac{\delta |W_i||W_j| \cdot \binom{|W_i| - 1}{k-1}\binom{|W_j| - 1}{k-1}}
%{\binom{|W_i|}{k}{\binom{|W_j|}{k}}} = \delta k^2 .
%\end{equation*}
\end{proof}

We now turn to the proof of Lemma \ref{lem:strong_reg}. This lemma is proved using a ``conditional'' regularity lemma for binary matrices, proved by Alon, Fischer and Newman in \cite{bipartite}. Let $A$ be an $n \times n$ matrix with $0/1$ entries whose rows and columns are indexed by $1,...,n$. For a pair of sets
$R,C \subseteq [n]$, we denote by $R \times C$ the submatrix of $A$ with rows from $R$ and columns from $C$ and we call it a {\em block}. The
{\em dominant value} of a block is the value, $0$ or $1$, that appears in at least half of the entries. For a constant $\delta < \frac{1}{2}$, we say that a block is $\delta$-homogeneous if its dominant value appears in at least a $(1 - \delta)$-fraction of the entries. Clearly, if $A$ is the adjacency matrix of a tournament $T$ and
$X,Y \subseteq V(T)$ are disjoint vertex sets, then $(X,Y)$ is $\delta$-homogeneous (in the tournament sense) if and only if the block
$X \times Y$ is $\delta$-homogeneous (in the matrix sense). Moreover, the dominant direction of $(X,Y)$ is $X \rightarrow Y$ if and only if the dominant value of $X \times Y$ is $1$.
%\footnote{If $X,Y$ are not disjoint then this is not necessarily true because we (arbitrarily) set the diagonal of $A$ to be $0$.}

The weight of a block $R \times C$ is defined as $\frac{|R||C|}{n^2}$.
Let $\mathcal{R} = \{R_1,...,R_s\}$ and $\mathcal{C} = \{C_1,...,C_t\}$ be partitions of $[n]$. We say that $(\mathcal{R},\mathcal{C})$ is a {\em $\delta$-homogeneous} partition of $A$ if
the total weight of non-$\delta$-homogeneous blocks $R_i \times C_j$ is at most $\delta$.
Note that if $A$ is the adjacency matrix of a tournament $T$ on $V(T) = [n]$, and if $\mathcal{P}$ is a partition of $[n]$ such that $(\mathcal{P},\mathcal{P})$ is a $\delta$-homogeneous partition of $A$, then $\mathcal{P}$ is a $\delta$-homogeneous partition of $T$ (as defined in the beginning of Section \ref{sec:easy})\footnote{The converse is not necessarily true. The fact that $\mathcal{P}$ is a $\delta$-homogeneous partition of $T$ does not take into account ``diagonal" blocks, i.e. blocks of the form $V_i \times V_i$, $V_i \in \mathcal{P}$.}.
% We note that the parts in $\mathcal{R}$ and $\mathcal{C}$ are allowed to be empty.

Let $B$ be a $0/1$-valued $k \times k$ matrix. A {\em copy} of $B$ in $A$ is a sequence of rows $r_1 < r_2 < \dots < r_k$ and a sequence of columns $c_1 < c_2 < \dots < c_k$ such that $A_{r_i,c_j} = B_{i,j}$ for every
$1 \leq i,j \leq k$. We are now ready to state the Alon-Fischer-Newman Regularity Lemma.
\begin{lemma}[{\bf Alon-Fischer-Newman \cite{bipartite}}] \label{lem:AFN}
There is $c > 0$ such that the following holds for every integer
$k \geq 1$ and every $\delta > 0$. For every $0/1$ matrix $A$ of size $n \times n$ with
$n > \left( k / \delta \right)^{ck}$, either $A$ has a $\delta$-homogeneous partition $(\mathcal{R},\mathcal{C})$ with $|\mathcal{R}|,|\mathcal{C}| \leq \left( k / \delta \right)^{ck}$, or for every $0/1$-valued $k \times k$ matrix $B$, $A$ contains at least
$\left( \delta / k \right)^{ck^2} n^{2k}$ copies of $B$.
\end{lemma}
Throughout this section $c$ denotes the constant from Lemma \ref{lem:AFN}. Without loss of generality, we always assume that $c \geq 1$. The following lemma is an application of Lemma \ref{lem:AFN} to adjacency matrices of tournaments.

\begin{lemma}\label{lem:reg}
Let $F = (M \cup N, E)$ be a bipartite tournament with $|M| = |N| = k$, and let
$\delta \in (0,1/2)$. Let $T$ be a tournament on
$n \geq n_0\left( k,\delta\right)$ vertices and let $\mathcal{P}$ be an equipartition of $V(T)$. Then either $T$ contains at least
$\left(\delta / 3k \right)^{2ck^2} n^{2k}$ copies of $F$, or $T$ admits a $\delta$-homogeneous equipartition that refines $\mathcal{P}$, and has at least $\delta^{-1}$ and at most
$|\mathcal{P}| \cdot \left( 3k / \delta \right)^{5ck}$ parts.
\end{lemma}
As the proof of Lemma \ref{lem:reg} is rather technical, we leave it to the end of this section and first show how to deduce Lemma \ref{lem:strong_reg} by two applications of Lemma \ref{lem:reg}.
\begin{proof}[Proof of Lemma \ref{lem:strong_reg}]
Let $C$ be large enough so that
\begin{equation*}
(\delta/k)^{C} \leq (\delta/15k)^{110c^2},
\end{equation*}
where $c \geq 1$ is the constant from Lemma \ref{lem:AFN}. Note that $C$ does not depend on $\delta$ or $k$, as $\delta$ is assumed to be less than $\frac{1}{2}$.

We assume that $T$ contains less than $(\delta/k)^{Ck^3}n^{2k}$ copies of $F$ and prove that the other alternative in the statement of the lemma holds. Our choice of $C$ implies that $T$ contains less than $(\delta/15k)^{2ck^2}n^{2k}$ copies of $F$.
By applying Lemma \ref{lem:reg} with approximation parameter $\frac{\delta}{5}$ and
$\mathcal{P} = \{V(T)\}$, we get that $T$ admits a
$\frac{\delta}{5}$-homogeneous equipartition $\Q = \{Q_1,...,Q_q\}$ with
\begin{equation}\label{eq:partition_size_1}
\delta^{-1} \leq q \leq (15k/\delta)^{5ck}.
\end{equation}

\noindent
Set $\gamma = \frac{1}{2q^4}$, and note that
$\gamma \geq \frac{1}{2}(\delta/15k)^{20ck}$, and hence
$$\left( \gamma/3k \right)^{2ck^2} \geq
\left( (1/6k) \cdot (\delta/15k)^{20ck} \right)^{2ck^2} \geq
\left( (\delta/15k)^{21ck} \right)^{2ck^2} \geq
(\delta/15k)^{42c^2k^3}.$$
Our assumption in the beginning of the proof and our choice of $C$ imply that $T$ contains less than
$\left( \gamma/3k \right)^{2ck^2} n^{2k}$ copies of $F$.
Apply Lemma \ref{lem:reg} to $T$ again, now with approximation parameter
$\gamma$ and $\mathcal{P} = \mathcal{Q}$, to obtain a
$\gamma$-homogeneous equipartition $\mathcal{W}$ which refines $\Q$ and satisfies
\begin{equation}\label{eq:partition_size_2}
\left| \W \right| \leq
\left| \mathcal{Q} \right| \cdot \left( 3k/\gamma \right)^{5ck} \leq
(15k/\delta)^{5ck} \cdot \left( 6k \cdot (15k/\delta)^{20ck} \right)^{5ck} \leq
(15k/\delta)^{110c^2k^2}.
\end{equation}

For each $1 \leq i \leq q$ define
$\mathcal{W}_i = \{W \in \mathcal{W} : W \subseteq Q_i\}$. Sample a vertex $w_i \in Q_i$ uniformly at random and let $W_i \in \mathcal{W}_i$ be such that
$w_i \in W_i$. By (\ref{eq:partition_size_2}) and our choice of $C$, we have
$|\mathcal{W}| \leq (k/\delta)^{Ck^2}$, which implies that
$|W_i| \geq (\delta/k)^{Ck^2}n$ for every $1 \leq i \leq q$, as required.
To complete the proof, we show that with positive probability, $W_1,...,W_q$ satisfy the assertions of Items 1 and 2 of the lemma.

Let $A_1$ be the event that $(W_i,W_j)$ is $\delta$-homogeneous for every $1 \leq i < j \leq q$.
Fixing $1 \leq i < j \leq q$, the probability that $(W_i,W_j)$ is not $\delta$-homogeneous is
$\sum{\frac{|W||W'|}{|Q_i||Q_j|}} =
\left( \frac{q}{n} \right)^2\sum{|W||W'|}$,
%\begin{equation*}
%\sum{\frac{|W||W'|}{|Q_i||Q_j|}} \leq
%\sum{\frac{|W||W'|}{\left( n/q \right)^2}} \leq
%2\sum{\frac{|W||W'|}{(n/q)^2}}
%\end{equation*}
where the sum is over all
non-$\delta$-homogeneous pairs
$(W,W') \in \mathcal{W}_i \times \mathcal{W}_j$. This sum is not larger than
$\gamma q^2 = \frac{1}{2q^2}$ because $\mathcal{W}$ is $\gamma$-homogeneous and by our choice of $\gamma$. By the union bound over all pairs
$1 \leq i < j \leq q$, we get that
$\mathbb{P}\left[ A_1 \right] \geq \frac{1}{2}$.

Let $1 \leq i < j \leq q$ be such that $(Q_i,Q_j)$ is a $\frac{\delta}{5}$-homogeneous pair. We say that $(Q_i,Q_j)$ is {\em bad} if
$d(Q_i,Q_j) \geq 1 - \frac{\delta}{5}$ but $d(W_i,W_j) \leq \delta$, or $d(Q_i,Q_j) \leq \frac{\delta}{5}$ but $d(W_i,W_j) \geq 1 - \delta$.  Otherwise $(Q_i,Q_j)$ is {\em good}. Assume without loss of generality that
$d(Q_i,Q_j) \geq 1 - \frac{\delta}{5}$. Then the probability that $d(W_i,W_j) \leq \delta$ is at most
$\frac{\delta/5}{1 - \delta} < \frac{2\delta}{5}$ (here we use $\delta < \frac{1}{2}$). We conclude that the probability that a given pair $(Q_i,Q_j)$ is bad is less than $\frac{2\delta}{5}$.
Let $Z$ be the number of bad pairs $(Q_i,Q_j)$.
Let $A_2$ be the event that $Z \leq \frac{4 \delta}{5}q^2$.
We have $\mathbb{E}[Z] < \frac{2\delta}{5}q^2$.
% such that the dominant value of $(W_i,W_j)$ is {\em not} the same as that of $(V_i,V_j)$.
By Markov's inequality, we have
$\mathbb{P}[Z > \frac{4\delta}{5}q^2] < \frac{1}{2}$, implying that
$\mathbb{P}[A_2] > \frac{1}{2}$.

So far we showed that with positive probability, both $A_1$ and $A_2$ happen. We now show that if $A_1$ and $A_2$ happen then Items 1 and 2 in the lemma hold. Item 2 holds because $A_1$ happened.
For Item 1, notice that if $(Q_i,Q_j)$ is $\frac{\delta}{5}$-homogeneous and good, and if $(W_i,W_j)$ is $\delta$-homogeneous, then $(W_i,W_j)$ has the same dominant direction as $(Q_i,Q_j)$. Thus, if a pair $1 \leq i < j \leq q$ violates Item 1, then either $(Q_i,Q_j)$ is not $\frac{\delta}{5}$-homogeneous or $(Q_i,Q_j)$ is bad. Since $\mathcal{Q}$ is a $\frac{\delta}{5}$-homogeneous equipartition, the number of non-$\frac{\delta}{5}$-homogeneous pairs $(Q_i,Q_j)$ is at most $\frac{\delta}{5}q^2$. Since $A_2$ happened, the number of bad pairs $(Q_i,Q_j)$ is at most $\frac{4\delta}{5}q^2$. Thus, Item 1 holds for all but at most $\delta q^2$ of the pairs $(Q_i,Q_j)$, as required.
\end{proof}

\begin{proof}[Proof of Lemma \ref{lem:reg}]
Let us assume that $T$ contains less than
$\left(\delta / 3k \right)^{2ck^2} n^{2k}$ copies of $F$. Our goal is to show that $T$ admits a $\delta$-homogeneous equipartition which refines $\mathcal{P}$, and has
at least $\delta^{-1}$ and at most $|\mathcal{P}| \cdot \left( 3k / \delta \right)^{5ck}$ parts.

%By removing at most
%$r$ vertices from each set $V_i$, we can get sets $V'_i \subseteq V_i$ such that $\left| V'_1 \right| = \left| V'_2 \right| \dots = \left| V'_D \right|$ and $\left| V'_i \right|$ is divisible by $r$. Let $T'$ be the tournament obtained by removing these vertices, i.e.
%$V(T') = V'_1 \cup \dots \cup V'_D$.
%Set
%$$m = |V(T')| \geq n - rD$$ and let us assume, without loss of generality, that
%$V(T') = [m]$.
%
%We will now show how to find a $\frac{\delta}{2}$-homogeneous equipartition $\Q$ of $T'$ into $r$ equal parts which refines
%$\mathcal{P}' := \{V'_1,\dots,V'_D\}$. Given this equipartition, we can obtain a $\delta$-homogeneous equipartition of $T$ which refines $\mathcal{P}$ by distributing the (at most $r$) vertices in $V_i \setminus V'_i$ as equally as possible among the parts of $\Q$ which are contained in $V'_i$ (for each $i = 1,\dots,D$), noticing that if $n$ is large enough then the resulting equipartition of $T$ would be $\delta$-homogeneous. Hence, our goal from now on is to prove that $T'$ admits a $\frac{\delta}{2}$-homogeneous equipartition into $r$ equals parts which refines $\mathcal{P}'$.

Define $B$ to be the {\em bipartite adjacency matrix} of $F$; that is, $B$ is a $k \times k$ matrix, indexed by $M \times N$, in which $B_{x,y} = 1$ if $(x,y) \in E(F)$ and $B_{x,y} = 0$ if $(y,x) \in E(F)$. We claim that
$A = A(T)$, the adjacency matrix of $T$, contains less than
$\left(\delta^2 / 3k \right)^{ck^2} {m}^{2k}$ copies of $B$. Assume otherwise. A copy of $B$ which does not intersect the main diagonal of $A$ corresponds to a copy of $F$ in $T$. There can be no more than $O(n^{2k-1})$ copies of $B$ which intersect the main diagonal of $A$. Recalling that
$\delta < \frac{1}{2}$ and assuming $n$ to be large enough, we conclude that $T$ contains at least
$\left(\delta^2 / 3k \right)^{ck^2} n^{2k} - O(n^{2k-1}) \geq
\left(\delta / 3k \right)^{2ck^2} n^{2k}$
copies of $F$, in contradiction to our assumption in the beginning of the proof.

Thus, $A = A(T)$ contains less than
$\left(\delta^2 / 3k \right)^{ck^2} n^{2k}$ copies of $B$.
By Lemma \ref{lem:AFN}, applied with approximation parameter $\frac{\delta^2}{3}$, $A$ has a $\frac{\delta^2}{3}$-homogeneous partition $(\mathcal{R},\mathcal{C})$ with
$|\mathcal{R}|,|\mathcal{C}| \leq \left( 3k / \delta^2 \right)^{ck}$.

Write $\mathcal{P} = \{P_1,...,P_p\}$.
For every $i = 1,\dots,p$, let $\U_i$ be the common refinement of the set $P_i$ and the partitions $\mathcal{R}$, $\mathcal{C}$, that is
$\U_i = \{P_i \cap R \cap C : R \in \mathcal{R}, C \in \mathcal{C}\}$.
Set $q = \frac{6p|\mathcal{R}| |\mathcal{C}|}{\delta}$, and note that
$\delta^{-1} \leq q \leq \frac{6p}{\delta} \cdot \left( 3k / \delta^2 \right)^{2ck} \leq
p \cdot \left( 3k / \delta \right)^{5ck}$.
For each $U \in \U_i$, partition $U$ into parts of size $\frac{n}{q}$ and an additional part $Z_{i,U}$ of size less than $\frac{n}{q}$. Let $Z_i$ be the union of all additional parts $Z_{i,U}$, $U \in \U_i$. Note that we have
$\left| Z_i \right| < |\mathcal{R}|\cdot|\mathcal{C}| \cdot \frac{n}{q} \leq
\frac{\delta n}{6p}$.
%\begin{equation}\label{eq:additional_part_bound}
%\left| Z_i \right| < |\mathcal{R}|\cdot|\mathcal{C}| \cdot \frac{n}{q} \leq
%\frac{\delta}{8p}.
%\end{equation}
Partition $Z_i$ arbitrarily into parts of size $\frac{n}{q}$. Denote by $\Q_i$ the resulting equipartition of $P_i$.
Then $\Q := \bigcup_{i=1}^{q}{\Q_i}$ is an equipartition of $V(T)$ which has $q$ parts and refines $\mathcal{P}$.

To finish the proof, we show that $\Q$ is $\delta$-homogeneous. To this end, define $\N$ to be the set of all non-$\delta$-homogeneous pairs $(X,Y) \in \Q \times \Q$ with $X \neq Y$.
Set $Z := \bigcup_{i=1}^{p}{Z_i}$ and note that
$\left| Z \right| = \sum_{i=1}^{p}{|Z_i|} \leq \frac{\delta}{6}n$. By the definition of $\Q$, if
$X \in \Q$ is not contained in $Z$ then there are $R \in \mathcal{R}$ and $C \in \mathcal{C}$ such that $X \subseteq R \cap C$. Thus, a block $X \times Y$ for which $X,Y \not\subseteq Z$ is contained in a block $R \times C$, where $R \in \mathcal{R}$ and $C \in \mathcal{C}$. Let $\N_1$ be the set of pairs $(X,Y) \in \N$ such that either $X$ or $Y$ is contained in $Z$; let $\N_2$ be the set of pairs $(X,Y) \in \N$ such that $X,Y \not\subseteq Z$, and the block $R \times C$ containing $X \times Y$ is not $\frac{\delta^2}{3}$-homogeneous; let $\N_3 = \N \setminus (\N_1 \cup \N_2)$. Since $|Z| \leq \frac{\delta}{6}n$, we have
$\sum_{(X,Y) \in \N_1}{|X||Y|} \leq
2n \cdot \left| Z \right| \leq \frac{\delta}{3}n^2$. Furthermore, as $(\mathcal{R},\mathcal{C})$ is $\frac{\delta^2}{3}$-homogeneous, we have
$\sum_{(X,Y) \in \N_2}{|X||Y|} \leq \frac{\delta^2}{3}n^2$. Thus, in order to prove that $\Q$ is $\delta$-homogeneous, it is enough to show that
$\sum_{(X,Y) \in \N_3}{|X||Y|} \leq \frac{\delta}{3}n^2$.
%\begin{equation}\label{eq:non_hom_parts_count_3}
%\sum_{(X,Y) \in \N_3}{|X||Y|} \leq \frac{\delta}{2}n^2.
%\end{equation}
%\begin{equation}\label{eq:non_hom_parts_count_1}
%\sum_{(X,Y) \in \N_1}{|X||Y|} \leq
%2n \cdot \left| Z \right| \leq \frac{\delta}{4}n^2.
%\end{equation}
%\begin{equation}\label{eq:non_hom_parts_count_1}
%\left| \N_1 \right| \leq 2 r \cdot \frac{|Z|}{m/r} \leq
%\frac{\delta^2}{8} r^2.
%\end{equation}

By the definition of $\N_3$, for every $(X,Y) \in \N_3$ there are $R \in \mathcal{R}$ and
$C \in \mathcal{C}$ such that the block $R \times C$ is $\frac{\delta^2}{3}$-homogeneous and contains the block $X \times Y$.
Let $R \times C$ be a $\frac{\delta^2}{3}$-homogeneous block, and
assume without loss of generality that the dominant value of $R \times C$ is $1$.
Let $X_1,...,X_a$ be the parts of $\Q$ that are contained in $R$, and let $Y_1,...,Y_b$ be the parts of $\Q$ that are contained in $C$.
Define $\gamma(i,j)$ to be the fraction of pairs
$(x,y) \in X_i \times Y_j$ for which $A_{x,y} = 0$. Obviously, if $\gamma(i,j) \leq \delta$ then the block $X_i \times Y_j$ is $\delta$-homogeneous.
% If $X_i \neq Y_j$ then $(X_i,Y_j)$ is also $\delta$-homogeneous in the tournament sense.
We have
\begin{equation*}
\sum_{i=1}^{a}\sum_{j=1}^{b}
{\frac{|X_i||Y_j|}{|R||C|} \gamma(i,j)}
\leq \frac{\delta^2}{3},
\end{equation*}
because the sum on the left hand side is a lower bound for the fraction of pairs
$(x,y) \in R \times C$ for which $A_{x,y} = 0$. By Markov's inequality we have
\begin{equation*}
\sum_{(i,j): \; \gamma(i,j) > \delta}\frac{|X_i||Y_j|}{|R||C|} \leq \frac{\delta}{3} .
\end{equation*}
%Therefore, the sum of $\frac{|X||Y|}{|R_p||C_q|}$ over all $(X,Y) \in \N_1$ for which $X \times Y \subseteq R_p \times C_q$ is at most $\frac{\delta}{2}$.
Thus, the sum of $|X||Y|$ over all pairs $(X,Y) \in \N_3$ for which
$X \times Y \subseteq R \times C$, is at most $\frac{\delta}{3}|R||C|$. By summing over all $\frac{\delta^2}{3}$-homogeneous blocks $R \times C$ of the partition $(\mathcal{R},\mathcal{C})$, we get
$$\sum_{(X,Y) \in \N_3}{|X||Y|} \leq
\sum_{R \in \mathcal{R}, C \in \mathcal{C}}{\frac{\delta}{3}|R||C|} \leq
\frac{\delta}{3}n^2,$$
as required.
\end{proof}

\section{Hard Tournaments}\label{sec:hard}

In this section we prove the second direction of Theorem \ref{thm:main}.
For convenience we restate it as follows.

\begin{theorem}\label{thm:hard}
For every $h$ there are $\varepsilon_0 = \varepsilon_0(h) > 0$ and
$\alpha = \alpha(h) > 0$ with the following property. For every non-$2$-colorable oriented graph $H$ on $h$ vertices, for every positive
$\varepsilon < \varepsilon_0$ and for every $n \geq n_0(\varepsilon)$ there is a tournament $T$ on $n$ vertices which is $\varepsilon$-far from being $H$-free but contains at most
$\varepsilon^{\alpha \log(1/\varepsilon)} n^h$ copies of $H$.
\end{theorem}

In this Section we prove Theorem \ref{thm:hard}.
Throughout this section, the vertex sets of all graphs and digraphs are assumed to be subsets of $\mathbb{N}$. The reason for this assumption is that sometimes we want to have a linear ordering of the vertices.
Before getting to the actual proof of Theorem \ref{thm:hard}, we first study some properties of homomorphisms between graphs which
take into account an order of their vertex sets.

\subsection{Order-preserving Homomorphisms}
\begin{definition}\label{def:ord_hom}(Order-Preserving Homomorphism)
Let $G,G'$ be (undirected) graphs. An {\em order-preserving homomorphism} from $G$ to $G'$ is a function $f : V(G) \rightarrow V(G')$ satisfying the following two conditions.
\begin{enumerate}
\item $f$ is order preserving: for every $i,j \in V(G)$, if
$i \leq j$ then $f(i) \leq f(j)$.
% \item For every $u \in V(T')$, the set $f^{-1}(u)$ is transitive.
\item $f$ is a graph homomorphism: for every $\{i,j\} \in E(G)$ we have $\{f(i),f(j)\} \in E(G')$.
\end{enumerate}
\end{definition}

We write $G' \homleq G$ if there is an order-preserving homomorphism from $G$ to $G'$. Notice that the relation $\homleq$ is transitive (the composition of order-preserving homomorphisms is also an order-preserving homomorphism). An {\em  order-preserving isomorphism} is an order-preserving homomorphism which is a graph isomorphism.
We write $G \cong G'$ if there is an order-preserving isomorphism between $G$ and $G'$. \footnote{Notice that two isomorphic labeled graphs may not have an order-preserving isomorphism between them. Moreover, if two graphs have an order-preserving isomorphism between them then it is unique, assuming that the vertices in each graph have different labels, which we always do in our setting.}

A subgraph of any graph $G$ is always assumed to inherit the same vertex-labeling as it had in $G$. The {\em ordered core} of $G$ is a smallest (with respect to number of vertices) subgraph of $G$ to which there is an order-preserving homomorphism from $G$. The ordered core of $G$ is assumed to inherit the same vertex-labeling as it had in $G$.
Notice that by definition, there is no order-preserving homomorphism from the ordered core of $G$ to a proper induced subgraph of it. We say that a graph is an {\em ordered core} if it is the ordered core of itself.

\begin{proposition}\label{prop:poset}
Let $G_1,G_2$ be a pair of ordered cores. If
$G_2 \homleq G_1$ and
$G_1 \homleq G_2$ then
$G_1 \cong G_2$.
\end{proposition}
\begin{proof}
%We only need to show that the relation $\homleq$ on $\mathcal{C}$ is antisymmetric. Let $(G_1,\sigma_1), (G_2,\sigma_2) \in \mathcal{G}$ and consider
%$C_1 = C(G_1,\sigma_1), C_2 = C(G_2,\sigma_2)$.
By assumption there exist order-preserving homomorphisms
%$f : C(G_1,\sigma_1) \rightarrow C(G_2,\sigma_2)$ and
%$g : C(G_2,\sigma_2) \rightarrow \nolinebreak C(G_1,\sigma_1)$.
$f : G_1 \rightarrow G_2$ and
$g : G_2 \rightarrow G_1$.
Then $g \circ f$ is an order-preserving homomorphism from $G_1$ to itself. Since $G_1$ is a core, $g$ must be surjective. The same argument shows that $f$ is surjective. So $f,g$ are bijections and since $f,g$ are order-preserving we have $g = f^{-1}$. Therefore $f,g$ are order-preserving graph isomorphisms, as required.
\end{proof}

\noindent
Proposition \ref{prop:poset} shows that the ordered core of a graph is unique up to order-preserving isomorphism.
% We note that the core is unique up to order-preserving isomorphism.
\begin{comment}
\begin{proposition}\label{prop:minimal}
Let $G$ be a graph and let $C$ be its ordered core. Then
for every order-preserving homomorphism
$f : G \rightarrow C$ it holds that $f|_{V(C)}$ is the identity map.
\end{proposition}
\begin{proof}
$f \circ f$ is an order-preserving homomorphism from $G$ to $C$. By the definition of an ordered core, $f \circ f$ must be onto $C$ and therefore $f|_{V(C)}$ is a bijection. Since $f$ is order-preserving, $f|_{V(C)}$ is the identity map.
\end{proof}
\end{comment}
%\noindent
%Proposition \ref{prop:minimal} shows that if $C$ is the ordered core of a graph $G$ then there is no order-preserving homomorphism from $C$ to any (not necessarily induced) subgraph of $C$.

\begin{proposition}\label{prop:minimal}
Let $G_1,G_2$ be a pair of ordered cores and suppose that $G_1 \cong G_2$. Then every order-preserving homomorphism $f : G_1 \rightarrow G_2$ is an order-preserving isomorphism.
\end{proposition}
\begin{proof}
By definition, there is an order-preserving isomorphism
$g : G_2 \rightarrow G_1$. Then $f \circ g$ is an order-preserving homomorphism from $G_2$ to itself. By the definition of an ordered core,
$f \circ g$ is a bijection. Since $f,g$ are order-preserving we have that $f \circ g$ is the identity map and hence
$f = g^{-1}$. So $f$ is an isomorphism, as required.
\end{proof}

Let $H$ be an oriented graph. We say that an edge $(i,j) \in E(H)$ is a
{\em forward-edge} if
$i < j$ and {\em backward-edge} (or {\em backedge}) otherwise.
% If $H$ is a transitive tournament then its {\em transitive ordering} is the ordering in which all edges are forward-edges.
The
{\em backedge graph} of $H$ is the (undirected) graph on $V(H)$ in which $\{i,j\}$ is an edge if and only if $i < j$ and $j \rightarrow i$. Note that the backedge graph depends on the labeling of the vertices of $H$. If we relabel the vertices of $H$ then we may get a different backedge graph. We will need the following characterization of oriented graphs with chromatic number at most $k$. Although this characterization will be crucial in our proof of Theorem \ref{thm:hard}, it is computationally inefficient even for $k = 2$, as shown in Section 5.
\begin{proposition}\label{prop:chromatic}
An oriented graph $H$ is $k$-colorable if and only if there is a labeling of the vertices of $H$ for which the corresponding backedge graph is $k$-colorable (as a graph).
\end{proposition}
\begin{proof}
Assume first that there is a labeling of $V(H)$ such that the corresponding backedge graph, $G$, has a proper (graph) $k$-coloring
$U_1 \cup \dots \cup U_k$. Then for every $i$, the set $U_i$ is acyclic (in $H$) because all the edges inside it are forward-edges.

Now assume that $H$ has a proper (acyclic) $k$-coloring
$U_1 \cup \dots \cup U_k$. For every $i = 1,...,k$, label the vertices of $U_i$ such that there are no backedges inside $U_i$ (this is clearly possible because $U_i$ is acyclic). Then $U_i$ is an independent set in the backedge graph corresponding to this labeling.
\end{proof}

For an oriented graph $H$ we define a family of graphs $\C = \C(H)$, all labeled with $[h]$, as follows. We go through all $h!$ vertex-labelings of $H$ using the labels $1,...,h$ (where $h = v(H)$), and for each labeling we take the ordered core of the corresponding backedge graph. Let $\mathcal{C}$ be the set
of all these ordered cores.
Proposition \ref{prop:poset} implies that $(\C,\homleq)$ is a poset in the following sense:
% $\homleq$ is a transitive relation and
for every $C_1,C_2 \in \C$, if $C_2 \homleq C_1$ and $C_1 \homleq C_2$ then $C_1 \cong C_2$. In other words, $\homleq$ is a partial order on the set of equivalence classes of $\C$ under the equivalence relation $\cong$.
Finally, let $K(H)$ be a maximal element of the poset $(\C,\homleq)$, i.e. $K(H)$ is an (arbitrary) element of a maximal equivalence class. The maximality of $K(H)$ implies that for every $C \in \C$, if there is an order-preserving homomorphism from $C$ to $K(H)$ (namely if $K(H) \homleq C$) then $C \cong K(H)$.
\begin{proposition}\label{prop:core}
Let $H$ be an oriented graph. Consider any vertex-labeling of $H$ and let $G$ be the corresponding backedge graph. For every order-preserving homomorphism
$f : \nolinebreak G \rightarrow \nolinebreak K(H)$ there is a set
$X \subseteq V(H)$ such that $f|_{X}$ is a (graph) isomorphism onto $K$.
\end{proposition}
\begin{proof}
Let $C$ be the ordered core of $G$. Then $f|_{V(C)}$ is an order-preserving homomorphism from $C$ to $K(H)$. By the maximality of $K(H)$ we have $C \cong K(H)$. Then $f|_{V(C)}$ is an order-preserving isomorphism by Proposition \ref{prop:minimal}, implying the assertion with $X = V(C)$.
%Let $g : K(H) \rightarrow C$ be an order-preserving isomorphism. Then $g \circ f$ is an order-preserving homomorphism from $G$ to $C$. By Proposition \ref{prop:minimal} we have that $(g \circ f) |_{V(C)}$ is the identity map, so
%the assertion follows with $X = V(C)$.
% If $F_{\sigma} \neq F_{\rho}$ then by the maximality of $F_{\rho}$ there is no order-preserving homomorphism from $(F_{\sigma},\sigma)$ to $(F_{\rho},\rho)$; this implies that there is no order-preserving homomorphism from
% If $F_{\sigma} = F_{\rho}$ then the assertion follows from Observation \ref{obs:minimal}.
\end{proof}
\begin{corollary}\label{cor:odd_cycle}
Let $H$ be a non-$2$-colorable oriented graph. Then the graph $K(H)$ contains a cycle $c_1c_2 \dots c_{\ell}c_1$ of length $\ell \geq 3$ with the following property. Consider any vertex-labeling of $H$ and let $G$ be the corresponding backedge graph. Then for every order-preserving homomorphism
$f : \nolinebreak G \rightarrow \nolinebreak K(H)$ there are vertices
$u_1 \in f^{-1}(c_1),\dots,u_{\ell} \in f^{-1}(c_{\ell})$ such that $u_1u_2\dots u_{\ell}u_1$ is a cycle in $G$.
\end{corollary}
\begin{proof}
By the definition of $K(H)$ there is a vertex-labeling of $H$ such that $K(H)$ is the ordered core of the corresponding backedge graph, $G_0$. By Proposition \ref{prop:chromatic}, $G_0$ is not $2$-colorable and therefore contains an odd cycle. It is easy to see that the homomorphic image of an odd cycle must contain an odd cycle. Therefore $K(H)$ contains an odd cycle, whose length is obviously at least $3$.
The other assertion of the corollary follows directly from Proposition \ref{prop:core}.
%Now, if $G$ is the backedge graph corresponding to some vertex-labeling of $H$ and $f : G \rightarrow K(H)$ is an order-preserving homomorphism then by Proposition \ref{prop:core} there is a subset $X \subseteq V(H)$ such that $f|_X$ is an isomorphism.
\end{proof}
\subsection{Proof of Theorem \ref{thm:hard}}

\noindent
The main ingredient in the proof of Theorem \ref{thm:hard} is the following construction (see \cite{RS} and \cite{subgraphs}).
\begin{theorem}\label{thm:RS}
For every $k \geq 3$ there are $\delta_0 = \delta_0(k)$ and $c = c(k)$ such that for every
$\delta < \delta_0$, for every $3 \leq \ell \leq k$ and for every sequence of distinct indices
$1 \leq i_1,i_2,\dots,i_{\ell} \leq k$,
% with $1 \leq i_j \leq k$
there is a graph
$R = R(k,\delta; i_1,\dots,i_{\ell})$ with the following properties:
\begin{enumerate}
\item $V(R) = X_1 \uplus ... \uplus X_{k}$ and $X_i$ is an independent set for every $i$.
\item $|V(R)| \geq \left( \frac{1}{\delta} \right)^{c\log(1/\delta)}$.
\item $E(R)$ is the union of at least $\delta |V(R)|^2$ pairwise edge-disjoint $k$-cliques, each of the form $\{x_1,\dots,x_k\}$ with $x_i \in X_i$.
% Moreover, there are no other $k$-cliques in $R$.
% each of the form $\{x_1,x_2,...,x_k\}$ where $x_i \in X_i$.
\item $R$ contains at most $|V(R)|^2$ cycles $x_{i_1}x_{i_2}...x_{i_{\ell}}x_{i_1}$ with $x_{i_j} \in X_{i_j}$ for $j = 1,\dots,\ell$.
\end{enumerate}
\end{theorem}
%In item 3, when we say that two copies of $K$ are edge-disjoint we mean that they have at most one vertex in common.
The proof of Theorem \ref{thm:RS} uses (simple variants of) Behrend's construction of a large set of integers without a $3$-term arithmetic progression (see \cite{Behrend}). As it is similar to related constructions proved in previous papers (see, e.g., \cite{subgraphs}) it is omitted.

\noindent
We are now ready to prove Theorem \ref{thm:hard}.
\begin{proof}[Proof of Theorem \ref{thm:hard}]
Let $H$ be a non-$2$-colorable oriented graph on $h$ vertices.
Consider the graph $K = K(H)$ defined in Subsection 4.1.
%By definition, there is a vertex-labeling of $H$ such that $K$ is the ordered core of the corresponding backedge graph, $G$.
Put $k = v(K)$ and write
$V(K) = \{a_1,...,a_k\}$, where
$1 \leq a_1 < \nolinebreak a_2 < \nolinebreak \dots < a_k \leq h$.\footnote{Recall that $K$ inherits its vertex-labeling from the backedge graph of $H$ whose ordered core is $K$ and whose vertex-labels are $1,\dots,h$.}
By Corollary \ref{cor:odd_cycle}, $K$ contains a cycle
%By the definition of an ordered core, there is an order-preserving homomorphism
%$f: G \rightarrow K$. By Proposition \ref{prop:chromatic}, $G$ is not $2$-colorable and therefore contains an odd cycle. It is easy to see that the homomorphic image of an odd cycle must contain an odd cycle. So $K$ contains an odd cycle
$(a_{i_1}a_{i_2}...a_{i_{\ell}}a_{i_1})$ of length $\ell \geq 3$.
Let $m_0 = m_0(h)$ and $\gamma = \gamma(h)$ be from Lemma \ref{lem:force}. Set $\varepsilon_0 = \varepsilon_0(h)$ to be small enough so that every $\varepsilon < \varepsilon_0$ will satisfy the inequalities
%\begin{equation}\label{eq:eps_0_hard}
%\log(1/\varepsilon) \geq \frac{\log(1/\gamma)}{1 - \gamma}, \; \; \;
%\varepsilon < \gamma^2, \; \; \;
%\varepsilon < \gamma \delta_0(k),
%\end{equation}
\begin{equation}\label{eq:eps_0_hard}
(\varepsilon \gamma^{-1})^{\log(\gamma/\varepsilon)} \leq \varepsilon^{0.5\log(1/\varepsilon)},
\; \; \;
\varepsilon < \gamma \delta_0(k),
\end{equation}
where $\delta_0(k)$ is from Theorem \ref{thm:RS}. Let
$\varepsilon < \varepsilon_0$. Let $R = R(k,\delta; i_1,\dots,i_{\ell})$ be the graph obtained by applying Theorem \ref{thm:RS} with parameters
$k$ and
\begin{equation*}
\delta = \varepsilon \gamma^{-1},
\end{equation*}
and with $i_1,\dots,i_{\ell}$ being the indices of the cycle in $K$ as above. Our choice of $\varepsilon_0$ guarantees that we can apply Theorem \ref{thm:RS} with the above $\delta$. Put
$r = |V(R)|$
%By permuting the clusters $X_1,...,X_k$ we may assume (by Theorem \ref{thm:RS}) that
%\begin{equation}\label{eq:cycle_count}
%\# \big\{ (x_{i_1},\dots,x_{i_{\ell}}) \in
%X_{i_1} \times \dots \times X_{i_{\ell}} :
%x_{i_1}x_{i_2}...x_{i_{\ell}}x_{i_1} \text{ is a cycle in } R
%\big\} \leq r^2
%\end{equation}
%$R$ contains at most $r^2$ cycles of the form $x_{i_1}x_{i_2}...x_{i_{\ell}}x_{i_1}$ with
%$x_{i_j} \in X_{i_j}$ (see Theorem \ref{thm:RS}).
and let $n$ be an integer which we assume, for simplicity of presentation, to be divisible by $r$. We will also assume that $n$ is large enough where needed.

% By Observation \ref{obs:chromatic}, the graph $G(H,\rho)$ is not $2$-colorable, so it contains an odd cycle.
%By the definition of $C(H,\rho)$, there is a graph homomorphism from $G(H,\rho)$ to $K$.

%We conclude that Theorem \ref{thm:RS} is applicable to $K$. For $i=1,...,k$, let $a_i = \rho^{-1}(i)$. Let $\varepsilon < \varepsilon_0(k)$ and consider the graph
%$R = R(K,\varepsilon)$ obtained by applying Theorem \ref{thm:RS} to $K$ with the vertex enumeration $a_1,...,a_k$.

By the definition of $K$, there is a vertex-labeling of $H$ such that $K$ is the ordered core of the corresponding backedge graph, $G_0$. Hence there is an order-preserving homomorphism $g : G_0 \rightarrow K$.
Denote $H_i = g^{-1}(a_i)$ for $i = 1,...,k$. We claim that $H_1,\dots,H_k$ have the following two properties.
\begin{enumerate}[(a)]
\item $H_i$ is acyclic for every $i = 1,...,k$.
\item For every $1 \leq i < j \leq k$, if $\{a_i,a_j\} \notin E(K)$ then
$H_i \rightarrow H_j$.
\end{enumerate}
% From the definition of order-preserving homomorphism it follows that: (a) $H_i$ is transitive for every $i = 1,...,k$, and (b)
Property (a) follows from the definition of a backedge graph and the fact that $g$ is a graph homomorphism. For property (b) we also need to use the fact that $g$ is order-preserving.
Define an oriented graph $D$ on $[k]$ as follows. For every $1 \leq i < j \leq k$, if $\{a_i,a_j\} \notin E(K)$ then $(i,j) \in D$ (that is, there is a directed edge from $i$ to $j$) and otherwise $(i,j),(j,i) \notin E(D)$. Then for every $(i,j) \in E(D)$ we have $H_i \rightarrow H_j$. So $H$ satisfies the conditions of Lemma \ref{lem:force} with respect to the $k$-coloring $V(H) = H_1 \cup \dots \cup H_k$ and the oriented graph $D$.
Apply Lemma \ref{lem:force} to get a $k$-partite tournament
$F = (V_1 \cup \dots \cup V_k, E(F))$ such that
$|V_i| = \frac{n}{r}$ (here we assume that $n$ is large enough so that $\frac{n}{r} \geq m_0(h)$), and
\begin{equation}\label{eq:edge_direction}
\forall 1 \leq i < j \leq k, \; \; \; \{a_i,a_j\} \notin E(K) \Longrightarrow V_i \rightarrow V_j.
\end{equation}

Let $\K$ be the collection of $k$-cliques given by Item 3 in Theorem \ref{thm:RS}. Note that $\left| \K \right| \geq \delta r^2$; that the $k$-cliques in $\K$ are pairwise edge-disjoint and of the form $\{x_1,\dots,x_k\}$ with $x_i \in X_i$; and that every edge in $R$ is contained in (exactly) one of these $k$-cliques. These properties will be important in what follows.
%, that is,
%\begin{equation}\label{eq:cliques_cover_edges}
%\forall \{y,z\} \in E(R) \; \;
%\exists \{x_1,\dots,x_k\} \in \K, \; \; \;
%y,z \in \{x_1,\dots,x_k\}.
%\end{equation}
We define a tournament $T$ on an $\frac{n}{r}$-blowup of $V(R)$, that is, each vertex $x \in V(R)$ is replaced by a vertex-set $B(x)$ of size $\frac{n}{r}$. Put
$B(X_i) = \bigcup_{x \in X_i}{B(x)}$.
The edges of $T$ are oriented as follows.
\begin{enumerate}
\item $B(X_i)$ is transitive for every $i = 1,...,k$.
\item For every $1 \leq i < j \leq k$ and for every
$x_i \in X_i, x_j \in X_j$, if $\{x_i,x_j\} \notin E(R)$ then set $B(x_i) \rightarrow B(x_j)$.
\item For every $\{x_1,...,x_k\} \in \K$, put a copy of $F$ on
$B(x_1) \cup \dots \cup B(x_k)$ with $B(x_i)$ playing the role of $V_i$ for every $i = 1,...,k$.
\end{enumerate}

Since every edge of $R$ is contained in one of the cliques in $\K$, Items $2$ and $3$ together define the orientation of edges between $B(y)$ and $B(z)$ for every pair of vertices $y,z \in V(R)$ which belong to different clusters $X_1,\dots,X_k$.
%Indeed, $\{y,z\} \notin E(R)$ then this pair is handled in Item 2, and if $\{y,z\} \in E(R)$ then there is (a unique) $x = (x_1,\dots,x_k) \in \K$ for which $y,z \in \{x_1,\dots,x_k\}$.
Therefore, Items 1-3 indeed define a tournament.
There is no contradiction in Item 3 because the cliques in $\K$ are pairwise edge-disjoint.
% , every pair of $k$-cliques in $R$ are edge-disjoint.

We now show that $T$ satisfies our requirements, that is, $T$ is $\varepsilon$-far from being $H$-free yet contains at most $\varepsilon^{\alpha\log(1/\varepsilon)}n^h$ copies of $H$ (for a constant $\alpha = \alpha(h)$ that we choose later). We start with the following two observations that play a central role in the proof. First, notice that by Item 2 and by the combination of Item $3$ and (\ref{eq:edge_direction}) we have the following:
\begin{equation}\label{eq:nonedges_oriented_forward}
\forall 1 \leq i < j \leq k, \; \; \; \{a_i,a_j\} \notin E(K) \Longrightarrow B(X_i) \rightarrow B(X_j).
\end{equation}
Secondly, let $\C$ be the set of all $\ell$-tuples $(v_{i_1},\dots,v_{i_{\ell}}) \in
B(X_{i_1}) \times \dots \times B(X_{i_{\ell}})$ such that for every $j = 1,\dots,\ell$, if $i_j < i_{j+1}$ then $v_{i_{j+1}} \rightarrow v_{i_j}$ and if $i_j > i_{j+1}$ then $v_{i_{j}} \rightarrow v_{i_{j+1}}$, with indices taken modulo $\ell$. We will need the inequality
\begin{equation}\label{eq:few_special_cycles}
\left| \C \right| \leq \frac{n^{\ell}}{r},
\end{equation}
which we now prove. Given $(v_{i_1},\dots,v_{i_{\ell}}) \in \C$, let $x_{i_j} \in X_{i_j}$ be such that $v_{i_j} \in B(x_{i_j})$. We claim that $x_{i_1}x_{i_2}\dots x_{i_{\ell}}x_{i_1}$ is a cycle in $R$. Let
$1 \leq j \leq \ell$ and let us first handle the case that
$i_{j} < \nolinebreak i_{j+1}$. By the definition of $\C$ we have $v_{i_{j+1}} \rightarrow v_{i_j}$. By Item 2 in the construction of $T$ above, if $\{x_{i_j},x_{i_{j+1}}\} \notin \nolinebreak E(R)$ then
$B(x_{i_j}) \rightarrow B(x_{i_{j+1}})$, in contradiction to
$v_{i_{j+1}} \rightarrow v_{i_j}$. Therefore
$\{x_{i_j},x_{i_{j+1}}\} \in \nolinebreak E(R)$ in this case. Similarly, if $i_{j} > i_{j+1}$ then by the definition of $\C$ we have $v_{i_{j}} \rightarrow v_{i_{j+1}}$. By Item 2 in the construction of $T$, if $\{x_{i_j},x_{i_{j+1}}\} \notin E(R)$ then
$B(x_{i_{j+1}}) \rightarrow B(x_{i_{j}})$, in contradiction to $v_{i_{j}} \rightarrow v_{i_{j+1}}$. Therefore $\{x_{i_j},x_{i_{j+1}}\} \in E(R)$ in this case as well,
proving our assertion that $x_{i_1}x_{i_2}\dots x_{i_{\ell}}x_{i_1}$ is a cycle in $R$. By Item $4$ in Theorem \ref{thm:RS}, there are at most $r^2$ such cycles in $R$. Since $T$ is an $\frac{n}{r}$-blowup of $R$ and
$\ell \geq 3$, we get that
$\left| \C \right| \leq r^2 \left( \frac{n}{r} \right)^{\ell} \leq \frac{n^{\ell}}{r}$, establishing (\ref{eq:few_special_cycles}).

Let us prove that $T$ contains at most
$\varepsilon^{\alpha \log(1/\varepsilon)} n^h$ copies of $H$, where
$\alpha = \alpha(h) = 0.5c$ and $c = c(k)$ is from Theorem \ref{thm:RS}. We will show that every copy of $H$ in $T$ contains vertices $v_{i_1},\dots,v_{i_{\ell}}$ such that $\left(v_{i_1},\dots,v_{i_{\ell}}\right) \in \C$ (recall the definition of $\C$ above). This will imply that $T$ contains at most $\left| \C \right| \cdot n^{h - \ell}$ copies of $H$. By (\ref{eq:few_special_cycles}) we have $\left| \C \right| \leq \frac{n^{\ell}}{r}$,
and by Item 2 in Theorem \ref{thm:RS} we have
$r \geq (1/\delta)^{c\log(1/\delta)}$. By using our choice of $\delta$ and the first inequality in (\ref{eq:eps_0_hard}), we will conclude that $T$ contains at most
\begin{equation*}
\frac{n^{\ell}}{r} \cdot n^{h - \ell} = \frac{n^h}{r} \leq
\delta^{c\log(1/\delta)}n^h =
\left(\varepsilon \gamma^{-1} \right)^{c\log(\gamma/\varepsilon)} n^h \leq
\varepsilon^{0.5c\log(1/\varepsilon)} n^h =
\varepsilon^{\alpha\log(1/\varepsilon)}n^h
\end{equation*}
copies of $H$, as required. Hence, in order to complete the proof of the theorem we only need to show that every copy of $H$ in $T$ contains vertices $v_{i_1},\dots,v_{i_{\ell}}$ such that $\left(v_{i_1},\dots,v_{i_{\ell}}\right) \in \C$.

Consider an embedding $\varphi : H \rightarrow T$; that is, $\Image \varphi$ is a copy of $H$ in $T$ with $\varphi(v)$ playing the role of $v$ for every $v \in V(H)$. For $i = 1,...,k$ define
$U_i = \varphi^{-1}(B(X_i))$. Then $U_i$ is acyclic by Item 1 in the construction of $T$ above. Consider the vertex-labeling of $H$ with labels $1,...,h$ in which: (a) for every $1 \leq i < j \leq k$, the labels given to the vertices of $U_i$ are smaller than the labels given to the vertices of $U_j$, and (b)  for every $i = 1,...,k$, the vertices in $U_i$ are labeled in such a way that all edges are forward-edges, that is, for every
$u,v \in U_i$ we have
$u \rightarrow v$ only if $u < v$ (such a vertex-labeling of $U_i$ exists since $U_i$ is acyclic). Let $G$ be the backedge graph of $H$ with respect to this vertex-labeling.
Notice that if $\{u,v\} \in E(G)$ and $u < v$ then there are
$1 \leq i < j \leq k$ such that $u \in U_i$ and $v \in U_j$, as $U_i$ is an independent set in $G$ for every $i = 1,\dots,k$.

We claim that the function
$f : V(H) \rightarrow V(K)$ which maps $U_i$ to $a_i$ is an order-preserving homomorphism from $G$ to $K$. The fact that $f$ is order-preserving is immediate from the definition of the labeling. To see that $f$ is a graph homomorphism, consider any edge
$\{u,v\} \in E \left( G \right)$ and assume without loss of generality that $u < v$. By the definition of a backedge graph we have $v \rightarrow u$ in $H$, implying that $\varphi(v) \rightarrow \varphi(u)$ in $T$. As mentioned before, there are
$1 \leq i < j \leq k$ such that
$u \in U_i$ and $v \in U_j$. Assume for the sake of contradiction, that we have
$\{f(u), f(v)\} = \{a_i,a_j\} \notin E(K)$. By (\ref{eq:nonedges_oriented_forward}), this implies that
$B(X_i) \rightarrow B(X_j)$. Since $\varphi(u) \in B(X_i)$ and
$\varphi(v) \in B(X_j)$, we get a contradiction to
$\varphi(v) \rightarrow \varphi(u)$.
% in contradiction to $B(X_i) \rightarrow B(X_j)$.
Therefore $\{f(u), f(v)\} = \{a_i,a_j\} \in E(K)$, showing that $f$ is a homomorphism.

Having shown that $f$ is an order-preserving homomorphism, we use
Corollary \ref{cor:odd_cycle} to infer that there are
$u_{i_j} \in f^{-1}(a_{i_j})$, $1 \leq j \leq \ell$, such that $u_{i_1}u_{i_2}\dots u_{i_{\ell}}u_{i_1}$ is a cycle in $G$. Denote $v_{i_j} = \varphi(u_{i_j})$ and observe that by the definition of $f$ we have $v_{i_j} \in B(X_{i_j})$. We now show that
$\left( v_{i_1},\dots,v_{i_{\ell}} \right) \in \C$.
%Let $x_{i_j} \in X_{i_j}$ be such that
%$\varphi(u_{i_j}) \in B(x_{i_j})$. We claim that $x_{i_1}x_{i_2}...x_{i_{\ell}}x_{i_1}$ is a cycle in $R$.
For every $1 \leq j \leq \ell$ we have $\{u_{i_j},u_{i_{j+1}}\} \in E\left( G \right)$ (with indices taken modulo $\ell$). Fix any
$1 \leq j \leq \ell$ and assume first that $i_j < i_{j+1}$.
% The case $i_j > i_{j+1}$ is symmetrical.
%$\{x_{i_{j}},x_{i_{j+1}}\} \in E(R)$, with indices taken modulo $\ell$.
%Let us assume that $i_{j} < i_{j+1}$. The other case is handled similarly.
Then
% we have $\{u_{i_j},u_{i_{j+1}}\} \in E\left( G \right)$ and therefore
$u_{i_{j+1}} \rightarrow u_{i_{j}}$ in $H$ by the definition of the backedge graph.
% and the assumption that $i_{j} < i_{j+1}$.
Therefore
$v_{i_{j+1}} = \varphi(u_{i_{j+1}}) \rightarrow \varphi(u_{i_j}) = v_{i_j}$, as $\varphi$ is an embedding.
Similarly, if $i_j > i_{j+1}$ then
% we have $\{u_{i_j},u_{i_{j+1}}\} \in E\left( G \right)$, implying that
$u_{i_{j}} \rightarrow u_{i_{j+1}}$ in $H$ by the definition of the backedge graph, implying that
$v_{i_{j}} = \varphi(u_{i_{j}}) \rightarrow \varphi(u_{i_{j+1}}) = v_{i_{j+1}}$.
This shows that $\left( v_{i_1},\dots,v_{i_{\ell}} \right) \in \C$, as required.

Having shown that $T$ contains at most
$\varepsilon^{\alpha \log(1/\varepsilon)} n^h$ copies of $H$, we now prove that $T$ is $\varepsilon$-far from being $H$-free. We say that an edge
$e$ is a {\em cluster-edge} if it is contained in $B(X_i)$ for some $i = 1,...,k$, and is a {\em cut-edge} otherwise. Let $T'$ be any tournament obtained from $T$ by reversing less than
$\varepsilon n^2$ edges. Our goal is to show that $T'$ contains a copy of $H$. Let $T''$ be the tournament that agrees with $T$ on all cut-edges and agrees with $T'$ on all cluster-edges. Then $T''$ and $T'$ disagree on less than $\varepsilon n^2$ edges, and the same is true for $T''$ and $T$.

For every
$Y = \{y_1,...,y_k\} \in \K$, the tournament
$T''[B(y_1) \cup \dots \cup B(y_k)]$ is a completion of $F$ (by Item 3 in the construction of $T$, and because $T''$ agrees with $T$ on cut-edges). By the choice of $F$ via Lemma \ref{lem:force}, $T''[B(y_1) \cup \dots \cup B(y_k)]$ contains a collection $\mathcal{H}(Y)$ of at least $\gamma \left( \frac{n}{r} \right)^2$ copies of $H$, any two of which do not share cut-edges.
Let
$Y = \{y_1,\dots,y_k\}$ and $Z = \{z_1,\dots,z_k\}$ be distinct cliques in $\K$. Since $Y$ and $Z$ are edge-disjoint, $T''[B(y_1) \cup \dots \cup B(y_k)]$ and
$T''[B(z_1) \cup \dots \cup B(z_k)]$ do not share cut-edges. Therefore, copies of $H$ from $\mathcal{H}(Y)$ do not share cut-edges with copies of $H$ from $\mathcal{H}(Z)$.
Put $\mathcal{H} := \bigcup_{Y \in \K}{\mathcal{H}(Y)}$. Then $\mathcal{H}$ is a collection of
copies of $H$ in $T''$, any two of which do not share cut-edges.
By $|\K| \geq \delta r^2$ and our choice of $\delta$, we have
$\left| \mathcal{H} \right| \geq \delta r^2 \gamma \left( \frac{n}{r} \right)^2 = \varepsilon n^2$.
Since the copies of $H$ in $\mathcal{H}$ do not share cut-edges, one must reverse at least $|\mathcal{H}| \geq \varepsilon n^2$ cut-edges in order to destroy all copies of $H$ in $T''$.
Recall that $T'$ and $T''$ agree on cluster-edges, and disagree on less than $\varepsilon n^2$ edges.
Therefore, one of the copies of $H$ in $T''$ is also present in $T'$.
This completes the proof of the theorem.
\end{proof}

\section{The Hardness of Deciding Tournament Colorability}\label{sec:np}
In this section we prove Theorem \ref{thm:NP_hard}. The main challenge in proving Theorem \ref{thm:NP_hard} is the case $k=2$.
\begin{theorem}\label{thm:NP_hard_k=2}
Deciding if a tournament is $2$-colorable is $NP$-hard.
\end{theorem}
After proving Theorem \ref{thm:NP_hard_k=2} we show how to deduce Theorem \ref{thm:NP_hard} from a simple reduction from the $k$-Colorability problem to the $(k-1)$-Colorability problem for every $k \geq 3$.
Theorem \ref{thm:NP_hard_k=2} is proved by showing a reduction from a known $NP$-hard problem: the {\em Triangle-Free Cut Problem}, to the Tournament $2$-Colorability problem.
\begin{definition}[{\bf Triangle-Free Cut}]
	For an (undirected) graph $G$, a {\em triangle-free cut} of $G$ is a $2$-coloring of $V(G)$ with no monochromatic triangle.
\end{definition}
\noindent
It is known that the problem of deciding if a given graph has a triangle-free cut is $NP$-hard (see \cite{triangle_cut}).

For a vertex $v$ in a tournament we denote $N^+(v) = \{u : v \rightarrow u\}$ and $N^-(v) = \{u : u\rightarrow v\}$. If a pair of vertices $u,v$ in a tournament satisfy $u \rightarrow v$ then we say that $u$
{\em dominates} $v$ and that $v$ is {\em dominated} by $u$.
For the proof of Theorem \ref{thm:NP_hard_k=2} we need the following proposition regarding the gadget $H$ depicted in Figure 1.
% Consider the gadget in Figure \ref{fig:H}.

\begin{proposition}\label{prop:gadget}
$H$ has the following properties.
	\begin{enumerate}
	\item $H $ has a proper 2-coloring in which $u$ and $v$ have the same color and all the vertices in the set
			$N^-(u) \cup N^+(v)$ have the other color.
			% different from the color of $u$ and $v$.
	\item In every proper 2-coloring of $H$, the colors of $u$ and $v$ are the same.	
	\end{enumerate}
\end{proposition}

\begin{proof}
	For Item 1, color $u,v,w$ with one color and $a,b,c,d$ with the other color. We now prove Item 2. Consider a $2$-coloring of $V(H)$ in which $u$ and $v$ have different colors, say $u$ is colored red and $v$ is colored blue. If there is a color, red or blue, that appears in both $\{a,b\}$ and $\{c,d\}$, then the coloring is not proper, as we get a monochromatic cyclic triangle by joining either $u$ or $v$. Therefore, we may assume that either $a,b$ are colored with red and $c,d$ are colored with blue, or vice versa. But in both cases there is no color for $w$ as $\{a,b,w\}$ and $\{c,d,w\}$ are cyclic triangles.
\end{proof}

\begin{figure}\label{fig:H}
\begin{center}
\begin{tikzpicture}
\draw[fill] (1,0) circle [radius=0.05];
\node [below right] at (1,0) {$v$};

\draw[fill] (-1,0) circle [radius=0.05];
\node [below right] at (-1,0) {$u$};

\draw[fill] (2,1) circle [radius=0.05];
\node [below right] at (2,1) {$d$};

\draw[fill] (2,2) circle [radius=0.05];
\node [below right] at (2,2) {$c$};

\draw[fill] (-2,1) circle [radius=0.05];
\node [above left] at (-2,1) {$b$};

\draw[fill] (-2,2) circle [radius=0.05];
\node [above left] at (-2,2) {$a$};

\draw[fill] (0,3) circle [radius=0.05];
\node [above right] at (0,3) {$w$};

% uv
\draw [->] (-1,0) -- (0.95,0);
%uw
\draw [->] (-1,0) -- (-0.02,2.94);
%wv
\draw [->] (0,3) -- (1-0.02,0.06);
%ud
\draw [->] (-1,0) -- (2-0.06,0.98);
%uc
\draw [->] (-1,0) -- (2-0.06,2-0.04);
%vd
\draw [->] (1,0) -- (2,1-0.08);
%vc
\draw [->] (1,0) -- (2-0.05,2-0.1);
%bu
\draw [->] (-2,1) -- (-1-0.05,0+0.05);
%au
\draw [->] (-2,2) -- (-1,0+0.1);
%bv
\draw [->] (-2,1) -- (1-0.12,0+0.04);
%av
\draw [->] (-2,2) -- (1-0.12,0+0.12);
%cd
\draw [->] (2,2) -- (2,1+0.05);
%ab
\draw [->] (-2,2) -- (-2,1+0.05);
%db
\draw [->] (2-0.15,1) -- (-2+0.15,1);
%da
\draw [->] (2-0.08,1+0.02) -- (-2+0.12,2-0.03);
%ca
\draw [->] (2-0.15,2) -- (-2+0.22,2);
%cb
\draw [->] (2-0.16,2-0.04) -- (-2+0.24,1+0.06);
%wc
\draw [->] (0+0.06,3-0.03) -- (2-0.06,2+0.03);
%dw
\draw [->] (2-0.05,1+0.05) -- (0+0.1,3-0.1);
%wa
\draw [->] (0-0.06,3-0.03) -- (-2+0.06,2+0.03);
%bw
\draw [->] (-2+0.05,1+0.05) -- (0-0.1,3-0.1);
\end{tikzpicture}
\caption{the gadget $H$}
\end{center}
\end{figure}
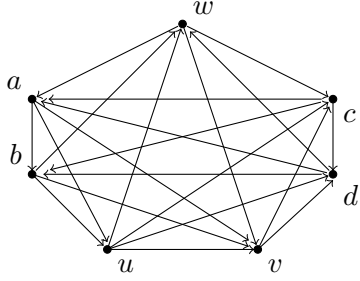

\begin{proof}[Proof of Theorem \ref{thm:NP_hard_k=2}]
	 Given a graph $G$ with vertices $V(G) = \{x_1,...,x_n\}$, we construct a tournament $T = T(G)$ and prove that $G$ has a triangle-free cut if and only if $T$ is $2$-colorable. $T$ is defined as follows.
	First, we put in $T$ vertices $y_1,...,y_n$ and set
	$y_i \rightarrow y_j$ for every
	$i < j$. We think of $y_i$ as corresponding to the vertex $x_i$ of $G$. Denote
	$Y = \{y_1,...,y_n\}$. Let $C_1,...,C_m$ be an enumeration of all triangles in $G$. Fix $1 \leq t \leq m$ and suppose that $C_t$ contains the vertices $x_i,x_j,x_k \in V(G)$, where $i < j < k$. We add to $T$ three new vertices, $z_t^{i},z_t^{j},z_t^{k}$, and set
	$z_t^{i} \rightarrow z_t^{j} \rightarrow z_t^{k} \rightarrow z_t^{i}$. So $Z_t := \left\{z_t^{i},z_t^{j},z_t^{k}\right\}$ spans a cyclic triangle. Set $Z_s \rightarrow Z_t$ for each $1 \leq s < t \leq m$. Denote
	$Z = \bigcup_{t=1}^{m}{Z_t}$ and set $Y \rightarrow Z$.
	
	Let $1 \leq t \leq m$, suppose that
	$Z_t = \left\{z_t^{i},z_t^{j},z_t^{k}\right\}$, where $i < j < k$, and fix any $\ell \in \{i,j,k\}$. We add a copy of $H$ (see Figure 1), denoted by $H_t^{\ell}$, in which $y_{\ell}$ plays the role of $u$, $z_t^{\ell}$ plays the role of $v$ and all other five vertices are new. Notice that this does not contradict $Y \rightarrow Z$, as we have
	$u \rightarrow v$ in $H$. Let $K_t^{\ell}$ be the subtournament of $H_t^{\ell}$ spanned by the five ``new'' vertices, that is
	$V(K_t^{\ell}) = V(H_t^{\ell}) \setminus (Y \cup Z)$. Set
	$K_t^{i} \rightarrow K_t^{j} \rightarrow K_t^{k}$ and
	$K_t^{i} \rightarrow K_t^{k}$.
	Denote
	$K_t = K_t^{i} \cup K_t^{j} \cup K_t^{k}$ and for each
	$1 \leq s < t \leq m$ set $K_s \rightarrow K_t$.
	
	Define $K = \bigcup_{t=1}^{m}{K_t}$ and note that we have $|Y| = n$, $|Z| = 3m$ and $|K| = 15m$. The vertex set of the tournament
	$T(G)$ is $Y \uplus Z \uplus K$. So far we defined the edges of $T(G)$ inside $Y$, $Z$ and $K$ and we set
	$Y \rightarrow Z$.
	We also already put some edges between $Y$ and $K$ and between $K$ and $Z$, namely the edges which are contained in $H_t^{\ell}$ for some
	$1 \leq t \leq m$ and $1 \leq \ell \leq n$. We direct all other edges from $Y$ to $K$ and from $K$ to $Z$; that is, if a pair
	$(p,q) \in Y \times K$ is not contained in any $H_t^{\ell}$ then we set $p \rightarrow q$, and similarly for $K$ and $Z$. In what follows we use the fact that an edge going from $K$ to $Y$ or from $Z$ to $K$ is contained in $H_t^{\ell}$ for some $1 \leq t \leq m$ and $1 \leq \ell \leq n$. This completes the definition of the tournament $T = T(G)$.
	
%	The only edges between $Y$ and $K$ and between $K$ and $Z$ that we already defined are edges which are contained in $H_t^{\ell}$ for some $1 \leq t \leq m$ and $\ell$. We now orient all other edges from $Y$ to $K$ and from $K$ to $Z$. In what follows we use the fact that an edge going from $K$ to $Y$ or from $Z$ to $K$ is contained in $H_t^{\ell}$ for some $1 \leq t \leq m$ and $\ell$.
	
	It remains to show that $G$ has a triangle-free cut if and only if $T$ is $2$-colorable. Assume first that $T$ admits a proper $2$-coloring,
		$c : V(T) \rightarrow \{\text{red, blue}\}$. For each $i = 1,...,n$ set $\phi(x_i) = c(y_i)$. We claim that $\phi$ is a triangle-free cut of $G$, that is, for every $1 \leq t \leq m$, the triangle $C_t$ in $G$ is not monochromatic. Fix $1 \leq t \leq m$ and suppose that $C_t$ contains the vertices $x_i,x_j,x_k$. By Item 2 in Proposition \ref{prop:gadget}, it must be the case that
		$c\left( z_t^{i} \right) = c(y_{i})$, $c\big( z_t^{j} \big) = c(y_{j})$ and $c\left( z_t^{k} \right) = c(y_{k})$. Since the set $Z_t = \{z_t^{i},z_t^{j},z_t^{k}\} \subseteq V(T)$ spans a cyclic triangle, we deduce that $c(y_i), c(y_j), c(y_k)$ are not all identical. Our choice of $\phi$ guarantees that $C_t$ is not monochromatic.
		
		Now assume that $G$ admits a triangle-free cut,
		$\phi : V(G) \rightarrow \{\text{red, blue}\}$. We define a $2$-coloring $c$ of $V(T)$ as follows. First, set $c(y_i) = \phi(x_i)$ for every $i = 1,\dots,n$. Next, let $1 \leq t \leq m$ and suppose that
		$Z_t = \left\{ z_t^{i},z_t^{j},z_t^{k} \right\}$. For each
		$\ell \in \{i,j,k\}$ set
		$c\left( z_t^{\ell} \right) = c(y_{\ell})$. Recall that $H_t^{\ell}$ is a copy of $H$ in which $y_{\ell}$ plays the role of $u$ and $z_t^{\ell}$ plays the role of $v$.
		Extend the coloring of $\left\{y_{\ell},z_t^{\ell}\right\}$ to a coloring of $H_t^{\ell}$ as in Item 1 of Proposition \ref{prop:gadget}, that is,  $H_t^{\ell}$ is colored properly and any vertex that dominates $y_{\ell}$ or that is dominated by $z_t^{\ell}$ has a different color from that of $y_{\ell}, z_t^{\ell}$. This guarantees that $H_t^{\ell}$ does not contain monochromatic edges going from $K$ to $Y$ or from $Z$ to $K$. As mentioned before, any edge in $T$ going from $K$ to $Y$ or from $Z$ to $K$ is contained in $H_t^{\ell}$ for some $1 \leq t \leq m$ and $1 \leq \ell \leq n$. We conclude that $T$ does not contain monochromatic edges going from $K$ to $Y$ or from $Z$ to $K$.
		
		It remains to show that the $2$-coloring $c$ of $V(T) = Y \cup Z \cup K$, defined in the previous paragraph, is proper. Let $S$ be a cyclic triangle in $T$. We show by case analysis that $S$ is not monochromatic. First we consider the cases (a) $S \subseteq Y \cup K$ and $S$ intersects both $Y$ and $K$, (b) $S \subseteq K \cup Z$ and $S$ intersects both $K$ and $Z$, (c) $S$ has one vertex in each of the sets $Y,Z,K$. Case (a) implies that $S$ contains an edge going from $K$ to $Y$. Similarly, case (b) implies that $S$ contains an edge that goes from $Z$ to $K$. Case (c) also implies that $S$ contains an edge from $Z$ to $K$ because $Y \rightarrow Z$. As proven in the previous paragraph, $T$ does not contain any monochromatic edge going from $K$ to $Y$ or from $Z$ to $K$. Therefore, $S$ is not monochromatic in each of the cases (a), (b) and (c).
		
		Given the previous paragraph, the only remaining cases to consider are
		$S \subseteq Y \cup Z$ and $S \subseteq K$. First, notice that the only cyclic triangles which are contained in $Z$ are
		$Z_1,\dots,Z_m$. Let $1 \leq t \leq m$ and suppose that
		$Z_t = \left\{ z_t^{i},z_t^{j},z_t^{k} \right\}$.

		By the definition of the coloring $c$ we have
		$c\left( z_t^{\ell} \right) = c(y_{\ell}) = \phi(x_{\ell})$ for every $\ell \in \{i,j,k\}$. The vertices of the triangle $C_t$ (in $G$) are $x_i,x_j,x_k$. Since $\phi$ is a triangle-free cut, it follows that $\phi(x_i), \phi(x_j), \phi(x_k)$ are not all identical. Therefore
		$c\left( z_t^{i} \right), c\big( z_t^{j} \big), c\left( z_t^{k} \right)$ are not all identical, namely $Z_t$ is not monochromatic.
		
		Recall that $Y$ is transitive and we have $Y \rightarrow Z$. Therefore $Y \cup Z$ does not contain any monochromatic cyclic triangle. Finally, every cyclic triangle inside $K$ is contained in some $K_t^{\ell}$. These triangles are not monochromatic because each $K_t^{\ell}$ is colored properly. This finishes the case analysis, showing that $T$ does not contain a monochromatic cyclic triangle and completing the proof of the theorem.
\end{proof}
\begin{proof}[Proof of Theorem \ref{thm:NP_hard}]
We will show that for every $k \geq 3$ there is a simple reduction from the $k$-Colorability problem to the $(k-1)$-Colorability problem. Given this reduction, we can prove the theorem by induction on $k$, with the base case $k=2$ already settled by Theorem \ref{thm:NP_hard_k=2}.

Let $T$ be a tournament. We define a tournament $T'$ as follows. The vertex-set of $T'$ consists of two vertex-disjoint copies of $T$, denoted $T_1$ and $T_2$, and an additional vertex $z$. We set
$T_1 \rightarrow T_2 \rightarrow z \rightarrow T_1$. We now show that $T$ is $(k-1)$-colorable if and only if $T'$ is $k$-colorable. First, if $T$ is $(k-1)$-colorable then clearly $T'$ is $k$-colorable: we color $T_1$ and $T_2$ according to a proper $(k-1)$-coloring of $T$, using the same $k-1$ colors for both $T_1$ and $T_2$, and then color $z$ with the remaining $k$'th color. It is easy to see that this $k$-coloring of $T'$ is proper. In the other direction, suppose that there is a proper coloring $c : V(T') \rightarrow [k]$ and assume without loss of generality that $c(z) = k$. Then it cannot be the case that both $T_1$ and $T_2$ contain a vertex with color $k$, as that will imply that there is a cyclic triangle in this color. Therefore, there is $i=1,2$ such that $T_i$ is colored with $[k-1]$, implying that $T$ is $(k-1)$-colorable. This completes the proof of the theorem.
\end{proof}

\end{document}